\documentclass[12pt]{amsart}
\usepackage{amsmath,amssymb,amsthm,longtable}
\usepackage{graphicx}
\usepackage{booktabs}
\usepackage[all]{xy}
\theoremstyle{plain}
\newtheorem{thm}{Theorem}[section]
\newtheorem{lem}[thm]{Lemma}
\newtheorem{cor}[thm]{Corollary}
\newtheorem{prop}[thm]{Proposition}
\newtheorem{defn}[thm]{Definition}
\newtheorem{fact}[thm]{Fact}
\newtheorem{example}[thm]{Example}

\newtheorem{rem}[thm]{Remark}

\newcommand{\bijarrow}{\mathrel{\stackrel{1:1}{\longleftrightarrow}}}
\newcommand{\ad}{\mathop{\mathrm{ad}}\nolimits}

\newcommand{\Id}{\text{id}}

\newcommand{\Exp}{\mathop{\mathrm{Exp}}\nolimits}
\newcommand{\Int}{\mathop{\mathrm{Int}}\nolimits}

\newcommand{\End}{\mathop{\mathrm{End}}\nolimits}
\newcommand{\Map}{\mathop{\mathrm{Map}}\nolimits}

\newcommand{\Ker}{\mathop{\mathrm{Ker}}\nolimits}

\newcommand{\rank}{\mathop{\mathrm{rank}}\nolimits}

\newcommand{\imag}{\text{im}}

\renewcommand{\theenumi}{\roman{enumi}}
\renewcommand{\labelenumi}{{\rm (}\theenumi{\rm )}}
\makeatletter
    
    \@addtoreset{equation}{section}
\usepackage{hyperref}
\makeatother
\begin{document}

\title{Smallest complex nilpotent orbits with real points}
\author{Takayuki Okuda}
\address{
Research Center for Pure and Applied Mathematics,
Graduate School of Information Sciences, Tohoku University,
6-3-09 Aoba, Aramakiaza, Aoba-ku, Sendai, Miyagi 980-8579 Japan}
\email{okuda@ims.is.tohoku.ac.jp}
\thanks{The author is supported by
Grant-in-Aid for JSPS Fellow No.25-6095}
\date{}
\maketitle 

\newcommand{\Orbit}{\mathcal{O}}
\newcommand{\R}{\mathbb{R}}
\newcommand{\C}{\mathbb{C}}
\newcommand{\minC}{\Orbit^{G_\C}_{\min}}
\newcommand{\ming}{\Orbit^{G_\C}_{\min,\mathfrak{g}}}

\begin{abstract}
Let $\mathfrak{g}$ be a non-compact real simple Lie algebra without complex structure,
and denote by $\mathfrak{g}_\mathbb{C}$ the complexification of $\mathfrak{g}$.
This paper focuses on non-zero nilpotent adjoint orbits in $\mathfrak{g}_\C$ meeting $\mathfrak{g}$.
We show that 
the poset consisting of such nilpotent orbits equipped with the closure ordering has the minimum $\ming$.
Furthermore, we determine such $\ming$ in terms of the Dynkin--Kostant classification
even in the cases where $\ming$ does not coincide with the minimal nilpotent orbit in $\mathfrak{g}_\C$.
We also prove that the intersection $\ming \cap \mathfrak{g}$ 
is the union of all minimal nilpotent orbits in $\mathfrak{g}$.
\end{abstract}


\section{Introduction and statement of main results}\label{section:intro}

Let $\mathfrak{g}$ be a non-compact real simple Lie algebra without complex structure. 
This means that the complexified Lie algebra $\mathfrak{g}_\mathbb{C}$ is simple.
Denote by $\mathcal{N}$ the nilpotent cone of $\mathfrak{g}_\mathbb{C}$ and by $\mathcal{N}/{G_\mathbb{C}}$ 
the set of complex nilpotent (adjoint) orbits of the group 
$G_\C := \Int (\mathfrak{g}_\C)$ of inner-automorphisms.

By abuse of notation, 
we write $\mathcal{N}_\mathfrak{g}/{G_\mathbb{C}}$ 
for the set consisting of complex nilpotent orbits that meet $\mathfrak{g}$. 
Note that 
\[
\mathcal{N}/{G_\mathbb{C}} \supset \mathcal{N}_\mathfrak{g}/{G_\mathbb{C}}.
\]
The finite sets $\mathcal{N}/{G_\mathbb{C}}$ and $\mathcal{N}_\mathfrak{g}/{G_\mathbb{C}}$ are both posets 
with respect to the closure ordering such that the zero-orbit $[0]$ is the minimum.
We ask what are minimal orbits in $(\mathcal{N}_\mathfrak{g}/{G_\mathbb{C}}) \setminus \{[0]\}$.
It is well known that 
$(\mathcal{N}/{G_\mathbb{C}}) \setminus \{[0]\}$ has the minimum $\minC$, 
which is called 
\emph{the minimal nilpotent orbit in $\mathfrak{g}_\mathbb{C}$}.
The minimal nilpotent orbit $\minC$ 
is the adjoint orbit that 
goes through a highest root vector with respect to a positive system 
$\Delta^+(\mathfrak{g}_\mathbb{C},\mathfrak{h}_\mathbb{C})$
where $\mathfrak{h}_\mathbb{C}$ is a Cartan subalgebra of $\mathfrak{g}_\mathbb{C}$ 
(see \cite[Chapter 4.3]{Collingwood-McGovern93} for the details).
In order to investigate $\mathcal{N}_\mathfrak{g}/{G_\mathbb{C}}$, 
we need a positive system $\Sigma^+(\mathfrak{g},\mathfrak{a})$ 
of the restricted root system of a maximally split abelian subspace $\mathfrak{a}$ of $\mathfrak{g}$
(see Section \ref{subsection:properties_of_restricted_roots} for the definition of maximally split abelian subspaces of $\mathfrak{g}$).

Our concern in this paper is with minimal orbits in $(\mathcal{N}_\mathfrak{g}/{G_\mathbb{C}}) \setminus \{[0]\}$.
Our first main result is here:

\begin{thm}\label{thm:O_min_g}
The following three conditions on 
a complex nilpotent orbit $\Orbit^{G_\mathbb{C}}$ in $\mathfrak{g}_\C$ 
with $\Orbit^{G_\C} \cap \mathfrak{g} \neq \emptyset$ are equivalent$:$
\begin{enumerate}
\item $\Orbit^{G_\mathbb{C}}$ is minimal in $(\mathcal{N}_\mathfrak{g}/{G_\mathbb{C}}) \setminus \{[0]\}$ with respect to the closure ordering.
\item The dimension of $\Orbit^{G_\mathbb{C}}$ attains its minimum in $(\mathcal{N}_\mathfrak{g}/{G_\mathbb{C}}) \setminus \{[0]\}$.
\item $\Orbit^{G_\mathbb{C}} \supset (\mathfrak{g}_\lambda \setminus \{ 0 \})$,
where $\lambda$ is the highest root of $\Sigma^+(\mathfrak{g},\mathfrak{a})$ 
and $\mathfrak{g}_\lambda$ is the root space of $\lambda$
$($the dimension of $\mathfrak{g}_\lambda$ is not necessary to be one$)$.
\end{enumerate}
Furthermore, there uniquely exists such $\Orbit^{G_\mathbb{C}}$ in $(\mathcal{N}_\mathfrak{g}/{G_\mathbb{C}}) \setminus \{[0]\}$.
\end{thm}

The unique complex nilpotent orbit in $\mathfrak{g}_\mathbb{C}$ in Theorem \ref{thm:O_min_g} will be denoted by $\Orbit^{G_{\mathbb{C}}}_{\min,\mathfrak{g}}$.
In many cases, $\ming = \minC$.

Our second main result concerns detailed properties of $\ming$ when $\ming \neq \minC$.

\begin{thm}\label{thm:WDD_of_our_orbits}
\renewcommand{\theenumi}{\arabic{enumi}}
\renewcommand{\labelenumi}{{\rm (}\theenumi{\rm )}}
\begin{enumerate}
\item \label{item:WDD:equiv_conditions_min_meets_g} 
The following five conditions on $\mathfrak{g}$ are equivalent$:$

\renewcommand{\theenumii}{\roman{enumii}}
\renewcommand{\labelenumii}{{\rm (}\theenumii{\rm )}}
\makeatletter
\renewcommand{\p@enumii}{}
\makeatother

\begin{enumerate}
\item \label{item:WDD:min=g_min} $\ming \neq \minC$.
\item \label{item:WDD:min_meets_g} $\minC \cap \mathfrak{g} = \emptyset$.
\item \label{item:WDD:one_dim_highest} $\dim_{\mathbb{R}} \mathfrak{g}_\lambda \geq 2$.
\item \label{item:WDD:extended} 
There exists a black node $\alpha$ in the Satake diagram of $\mathfrak{g}$ such that $\alpha$ has some edges connected to the added node in the extended Dynkin diagram of $\mathfrak{g}_\mathbb{C}$.
\item \label{item:WDD:classification} $\mathfrak{g}$ is isomorphic to one of 
$\mathfrak{su}^*(2k)$, $\mathfrak{so}(n-1,1)$, $\mathfrak{sp}(p,q)$, 
$\mathfrak{f}_{4(-20)}$ or $\mathfrak{e}_{6(-26)}$ where $k \geq 2, n \geq 5$ and $p, q \geq 1$.
\end{enumerate}
\item \label{item:Table} If the above equivalent conditions on $\mathfrak{g}$ are satisfied, 
then the complex nilpotent orbit $\ming$ is characterized by the weighted Dynkin diagram in Table \ref{table:O_min_g} via the Dynkin--Kostant classification.
\begin{table}[htb]
\begin{center}
\caption{List of $\ming$ for the cases $\ming \neq \minC$.}
	\begin{tabular}{llll} \hline \hline
	$\mathfrak{g}$ & $\dim_\mathbb{C} \ming$ & Weighted Dynkin diagram of $\ming$ \\ \midrule
	$\mathfrak{su}^*(2k)$ & $8k-8$ & \begin{xy}
	*++!D{0} *\cir<2pt>{}        ="A",
	(6,0) *++!D{1} *\cir<2pt>{} ="B",
	(12,0), *++!D{0} *\cir<2pt>{} ="C",
	(18,0), *++!D{0} *\cir<2pt>{} ="D",
	(24,0) = "D_1",
	(30,0) = "D_2",
	(36,0), *++!D{0} *\cir<2pt>{} ="E",
	(42,0), *++!D{0} *\cir<2pt>{} ="F",
	(48,0) *++!D{1} *\cir<2pt>{} ="G",
	(54,0) *++!D{0} *\cir<2pt>{} ="H",
	\ar@{-} "A";"B"
	\ar@{-} "B";"C"
	\ar@{-} "C";"D"
	\ar@{-} "D";"D_1"
	\ar@{.} "D_1";"D_2" ^*U{\cdots}
	\ar@{-} "D_2";"E"
	\ar@{-} "E";"F"
	\ar@{-} "F";"G"
	\ar@{-} "G";"H"
\end{xy} $( k \geq 3 )$ \\ 
& & \begin{xy}
	*++!D{0} *\cir<2pt>{}        ="A",
	(6,0) *++!D{2} *\cir<2pt>{} ="B",
	(12,0), *++!D{0} *\cir<2pt>{} ="C",
	\ar@{-} "A";"B"
	\ar@{-} "B";"C"
\end{xy} $(k = 2)$ \\ \hline
	$\mathfrak{so}(n-1,1)$ & $2n-4$ & \begin{xy}
	*++!D{2} *\cir<2pt>{}        ="A",
	(6,0) *++!D{0} *\cir<2pt>{} ="B",
	(12,0) *++!D{0}     *\cir<2pt>{} ="C",
	(18,0) = "C_1",
	(24,0) = "C_2",
	(30,0) *++!D{0} *\cir<2pt>{} ="D",
	(36,0) *++!D{0} *\cir<2pt>{} ="E",
	\ar@{-} "A";"B"
	\ar@{-} "B";"C"
	\ar@{-} "C";"C_1"
	\ar@{.} "C_1";"C_2" ^*U{\cdots}
	\ar@{-} "C_2" ; "D"
	\ar@{=>} "D";"E"
\end{xy} \quad $($\text{$n$ is odd, $n \geq 5$}$)$ \\
& & \begin{xy}
	*++!D{2} *\cir<2pt>{}        ="A",
	(6,0) *++!D{0} *\cir<2pt>{} ="B",
	(12,0) *++!D{0}     *\cir<2pt>{} ="C",
	(18,0) = "C_1",
	(24,0) = "C_2",
	(30,0) *++!D{0} *\cir<2pt>{} ="D",
	(36,3) *++!D{0} *\cir<2pt>{} ="E",
	(36,-3) *++!D{0} *\cir<2pt>{} ="F",
	\ar@{-} "A";"B"
	\ar@{-} "B";"C"
	\ar@{-} "C";"C_1"
	\ar@{.} "C_1";"C_2" ^*U{\cdots}
	\ar@{-} "C_2" ; "D"
	\ar@{-} "D";"E"
	\ar@{-} "D";"F"
\end{xy} \quad $($\text{$n$ is even, $n \geq 6$}$)$ \\ \hline
	$\mathfrak{sp}(p,q)$ & $4(p+q)-2$ & \begin{xy}
	*++!D{0} *\cir<2pt>{}        ="A",
	(6,0) *++!D{1} *\cir<2pt>{} ="B",
	(12,0) *++!D{0} *\cir<2pt>{} ="C",
	(18,0) *++!D{0} *\cir<2pt>{} ="D",
	(24,0) = "D_1",
	(30,0) = "D_2",
	(36,0) *++!D{0} *\cir<2pt>{} ="E",
	(42,0) *++!D{0} *\cir<2pt>{} ="F",
	\ar@{-} "A";"B"
	\ar@{-} "B";"C"
	\ar@{-} "C";"D"
	\ar@{-} "D";"D_1"
	\ar@{.} "D_1";"D_2" ^*U{\cdots}
	\ar@{-} "D_2" ; "E"
	\ar@{<=} "E";"F"
\end{xy} $(p+q \geq 3,~p, q \geq 1)$ \\ 
 & & \begin{xy}
	*++!D{0} *\cir<2pt>{}        ="A",
	(6,0) *++!D{2} *\cir<2pt>{} ="B",
	\ar@{<=} "A";"B"
\end{xy} $(p=q=1)$ \\ \hline
	$\mathfrak{e}_{6(-26)}$ & $32$ & 
\begin{xy}
	*++!D{1} *\cir<2pt>{}        ="A",
	(6,0) *++!D{0} *\cir<2pt>{} ="B",
	(12,0) *++!D{0}  *\cir<2pt>{} ="C",
	(18,0) *++!D{0} *\cir<2pt>{} ="D",
	(24,0) *++!D{1} *\cir<2pt>{} ="E",
	(12,-6) *++!L{0} *\cir<2pt>{} ="F",
	\ar@{-} "A";"B"
	\ar@{-} "B";"C"
	\ar@{-} "C";"D"
	\ar@{-} "D";"E"
	\ar@{-} "C";"F"
\end{xy} \\ \hline
	$\mathfrak{f}_{4(-20)}$ & $22$ & 
\begin{xy}
	*++!D{0} *\cir<2pt>{}        ="A",
	(6,0) *++!D{0} *\cir<2pt>{} ="B",
	(12,0) *++!D{0}  *\cir<2pt>{} ="C",
	(18,0) *++!D{1} *\cir<2pt>{} ="D",
	\ar@{-} "A";"B"
	\ar@{=>} "B";"C"
	\ar@{-} "C";"D"
\end{xy} \\ 
	\hline \hline
\label{table:O_min_g}
	\end{tabular}
\end{center}
\end{table}
\end{enumerate}
\end{thm}

\renewcommand{\theenumi}{\roman{enumi}}
\renewcommand{\labelenumi}{{\rm (}\theenumi{\rm )}}

The equivalence between \eqref{item:WDD:min_meets_g} and \eqref{item:WDD:classification} in Theorem \ref{thm:WDD_of_our_orbits} \eqref{item:WDD:equiv_conditions_min_meets_g} was stated on Brylinski \cite[Theorem 4.1]{Brylinski98} without proof.
We provide a proof for the convenience of the readers. 

Our work is motivated by the recent progress 
in the theory of infinite dimensional representations.
For an irreducible (admissible) representation $\pi$ of a real reductive Lie group $G$ with its Lie algebra $\mathfrak{g}$,
one can define 
the associated variety $\mathcal{AV}(\mathrm{Ann}~\pi) (\subset \mathcal{N})$ 
of the annihilator $\mathrm{Ann}~\pi$ in the enveloping algebra $\mathcal{U}(\mathfrak{g}_\C)$.
It is known that there uniquely exists a complex nilpotent orbit $\Orbit^{G_\C}_\pi$ meeting $\mathfrak{g}$ such that $\mathcal{AV}(\mathrm{Ann}~\pi) = \overline{\Orbit^{G_\C}_\pi}$.
Half the complex dimension of $\Orbit^{G_\C}_\pi$ coincides with the Gelfand--Kirillov dimension of $\pi$.

If $\pi$ is a minimal representation in the sense that the annihilator of $\pi$ is the Joseph ideal (\cite{Joseph76}), 
then $\mathcal{AV}(\mathrm{Ann}~\pi) = \overline{\minC}$ (Vogan \cite{Vogan91}).
Hence $\minC \cap \mathfrak{g} \neq \emptyset$, 
or equivalently $\minC \cap \mathfrak{p}_\C \neq \emptyset$
by the Kostant--Sekiguchi correspondence,
where $\mathfrak{g} = \mathfrak{k} + \mathfrak{p}$ is 
a Cartan decomposition of $\mathfrak{g}$ and 
$\mathfrak{p}_\C$ denotes the complexification of $\mathfrak{p}$.
Therefore, minimal representations do not exist for simple Lie groups $G$ 
if $\minC \cap \mathfrak{g} = \emptyset$ or equivalently,
if the Lie algebra $\mathfrak{g}$ of $G$ 
is one of the five simple Lie algebras in Theorem \ref{thm:WDD_of_our_orbits}.

In the cases where $\minC \cap \mathfrak{g} = \emptyset$ or equivalently $\ming \neq \minC$,
there is no minimal representation of $G$, 
however, Hilgert, Kobayashi and M\"ollers \cite{Hilgert-Kobayashi-Mollers:minimal_representation} 
recently constructed the ``smallest'' irreducible unitary representations $\pi$ of certain families of reductive Lie groups $G$.
They proved that $\mathcal{AV}(\mathrm{Ann}~\pi) = \overline{\ming}$ for their representations $\pi$. 
Therefore, $\pi$ attains the minimum of the Gelfand--Kirillov dimensions of infinite dimensional irreducible representations of $G$.
They constructed an $L^2$-model of such representations on a 
Lagrangian subvariety of minimal nilpotent $G$-orbits in $\mathfrak{g}$.
Our results here were used in \cite[Section 2.1.3]{Hilgert-Kobayashi-Mollers:minimal_representation}
in their proof that their representation $\pi$
attains the minimum of the Gelfand--Kirillov dimension
of all infinite dimensional irreducible (admissible) representations of $G$.

Our work is also related to \cite[Corollary 5.9]{Kobayashi-Oshima2013-gK} by Kobayashi and Oshima, 
on the classification of reductive symmetric pairs $(\mathfrak{g},\mathfrak{h})$ 
for which there exists a $(\mathfrak{g},K)$-module that is discretely decomposable as an $(\mathfrak{h}, H \cap K)$-module in the sense of \cite{Kobayashi97discreteIII}.

With applications to representation theory in mind,
we also study the intersection $\ming \cap \mathfrak{g}$
as a union of real nilpotent (adjoint) orbits in $\mathfrak{g}$
in this paper.

We denote the nilpotent cone for $\mathfrak{g}$ 
by 
\[
\mathcal{N}(\mathfrak{g}) := \mathcal{N} \cap \mathfrak{g}
\]
and by $\mathcal{N}(\mathfrak{g})/G$ 
the set of real nilpotent (adjoint) orbits in $\mathfrak{g}$
by the group $G := \Int (\mathfrak{g})$ of inner-automorphisms.
The set $\mathcal{N}(\mathfrak{g})/{G}$ is a poset with respect to the closure ordering, where the zero-orbit $[0]$ in $\mathfrak{g}$ is the minimum in $\mathcal{N}(\mathfrak{g})/G$.

For each real nilpotent orbit $\Orbit^G$ in $\mathfrak{g}$,
there exists the unique complex nilpotent orbit $\Orbit^{G_\mathbb{C}}$ in $\mathfrak{g}_\mathbb{C}$ which contains $\Orbit^G$.
Then $\Orbit^G$ is a real form of $\Orbit^{G_\mathbb{C}}$.
The correspondence $\Orbit^{G}$ to $\Orbit^{G_\mathbb{C}}$ gives a surjective map 
\[
\mathcal{N}(\mathfrak{g})/{G} \twoheadrightarrow \mathcal{N}_\mathfrak{g}/{G_\mathbb{C}} \ (\subset \mathcal{N}/{G_\C}).
\]
We remark that this map needs not be injective.
It is known that for a complex nilpotent orbit $\Orbit^{G_\mathbb{C}}$ in $\mathcal{N}_{\mathfrak{g}}/{G_\mathbb{C}}$,
the intersection $\Orbit^{G_\mathbb{C}} \cap \mathfrak{g}$ split into finitely many real nilpotent orbits.

Our third main result is a characterization of minimal orbits in $(\mathcal{N}(\mathfrak{g})/{G}) \setminus \{[0]\}$ as real forms of $\ming$.
More precisely, we prove the next theorem:

\begin{thm}\label{thm:real_minimal}
The following four conditions on a real nilpotent orbit $\Orbit^G$ in $\mathfrak{g}$ are equivalent$:$
\begin{enumerate}
\item \label{item:real_minimal:real_mini} $\Orbit^G$ is minimal in $(\mathcal{N}(\mathfrak{g})/{G}) \setminus \{[0]\}$ with respect to the closure ordering.
\item \label{item:real_minimal:real_dim_mini} The dimension of $\Orbit^G$ attains its minimum in $(\mathcal{N}(\mathfrak{g})/{G}) \setminus \{ [0] \}$.
\item \label{item:real_minimal:cpx} $\Orbit^G$ is contained in $\ming$.
\item \label{item:real_minimal:real-form} $\Orbit^G$ is a real form of $\ming$.
\end{enumerate}
In particular, the real dimension of such $\Orbit^G$ is equal to the complex dimension of $\ming$.
\end{thm}

We say that a real nilpotent orbit $\Orbit^G$ in $\mathfrak{g}$ is \emph{minimal} if $\Orbit^G$ satisfies the equivalent conditions in Theorem \ref{thm:real_minimal}.
In this sense, the intersection $\ming \cap \mathfrak{g}$ is the disjoint union of all minimal real nilpotent orbits in $\mathfrak{g}$.

Our fourth main result is to determine the number of minimal real nilpotent orbits in $\mathfrak{g}$ as follows:

\begin{thm}\label{thm:number_of_G-orbits_in_minimal}
For a non-compact real simple Lie algebra $\mathfrak{g}$ without complex structure, 
\begin{multline*}
\sharp \{\text{ minimal real nilpotent orbits in $\mathfrak{g}$ }\} \\
= 
\begin{cases} 
1 \quad \text{if $(\mathfrak{g},\mathfrak{k})$ is of non-Hermitian type,} \\ 
2 \quad \text{if $(\mathfrak{g},\mathfrak{k})$ is of Hermitian type,}
\end{cases}
\end{multline*}
where $\mathfrak{g} = \mathfrak{k} + \mathfrak{p}$ is a Cartan decomposition of $\mathfrak{g}$. 
\end{thm}

Note that all of five real simple Lie algebras $\mathfrak{g}$ in Theorem $\ref{thm:WDD_of_our_orbits}$ are of non-Hermitian type,
and therefore there uniquely exists a minimal real nilpotent orbits in such $\mathfrak{g}$.

In our proof of Theorem \ref{thm:number_of_G-orbits_in_minimal},
we study $MA$-orbits in a highest root space $\mathfrak{g}_\lambda$ of $\mathfrak{g}$ 
(see Section \ref{section:real-orbits_in_min} for the notation of a group $MA$ and more details).

\newcommand\twoheaduparrow{\mathrel{\rotatebox[origin=c]{90}{$\twoheadrightarrow$}}}

Our results on real nilpotent orbits (Theorem \ref{thm:real_minimal} and \ref{thm:number_of_G-orbits_in_minimal}) yield those on $K_\mathbb{C}$-orbits on the nilpotent cone $\mathcal{N}(\mathfrak{p}_\mathbb{C})$ via the Kostant--Sekiguchi correspondence.
To be precise, we fix some notation.

Let $\mathfrak{g}= \mathfrak{k} + \mathfrak{p}$ be a Cartan decomposition of $\mathfrak{g}$,
denote its complixification by $\mathfrak{g}_\mathbb{C} = \mathfrak{k}_\mathbb{C} + \mathfrak{p}_\mathbb{C}$,
and $K_\mathbb{C}$ the connected complex subgroup of $\Int(\mathfrak{g}_\mathbb{C})$ with its Lie algebra $\mathfrak{k}_\mathbb{C}$.
We define the nilpotent cone for $\mathfrak{p}_\mathbb{C}$ by \[
\mathcal{N}(\mathfrak{p}_\mathbb{C}) := \mathcal{N} \cap \mathfrak{p}_\mathbb{C},
\]
on which $K_\mathbb{C}$ acts with finitely many orbits.
The Kostant--Sekiguchi correspondence is a bijection between two finite sets
\[
\mathcal{N}(\mathfrak{g})/{G} \bijarrow \mathcal{N}(\mathfrak{p}_\mathbb{C})/{K_\mathbb{C}}
\]
which preserves the closure ordering (Barbasch--Sepanski \cite{Barbasch-Sepanski98}).

Let us denote by $\mathcal{N}_{\mathfrak{p}_\mathbb{C}}/{G_\mathbb{C}}$ 
the subset of $\mathcal{N}/{G_\mathbb{C}}$ consisting of complex nilpotent orbits in $\mathfrak{g}_\C$ meeting $\mathfrak{p}_\C$.
Recall that 
a complex nilpotent orbit $\Orbit^{G_\mathbb{C}}$ in $\mathfrak{g}_\mathbb{C}$ meets $\mathfrak{g}$
if and only if it meets $\mathfrak{p}_\mathbb{C}$ (Sekiguchi \cite[Proposition 1.11]{Sekiguchi84}).
Thus $\mathcal{N}_{\mathfrak{p}_\mathbb{C}}/{G_\mathbb{C}}$ coincides with $\mathcal{N}_{\mathfrak{g}}/{G_\mathbb{C}}$ as a subset of $\mathcal{N}/G_\C$.

Further, for each nilpotent $K_\C$-orbit $\Orbit^{K_\mathbb{C}}$ in $\mathcal{N}(\mathfrak{p}_\C)/{K_\C}$, 
there uniquely exists a complex nilpotent $G_\C$-orbit $\Orbit^{G_\mathbb{C}}$ containing $\Orbit^{K_\mathbb{C}}$,
and the correspondence $\Orbit^{K_\mathbb{C}}$ to $\Orbit^{G_\mathbb{C}}$ gives a surjection map 
\[
\mathcal{N}(\mathfrak{p}_\mathbb{C})/{K_\mathbb{C}}
 \twoheadrightarrow  \mathcal{N}_{\mathfrak{p}_{\C}}/{G_\mathbb{C}} \ (\subset \mathcal{N}/{G_\C}).
\]
Thus, the Kostant--Sekiguchi correspondence gives the following commutative diagram:
\begin{align*}
\begin{array}{ccc}
\mathcal{N}_\mathfrak{g}/{G_\mathbb{C}} &= & \mathcal{N}_\mathfrak{p_\mathbb{C}} /{G_\mathbb{C}} \\ 
\twoheaduparrow & & \twoheaduparrow \\
\mathcal{N}(\mathfrak{g})/{G} & \bijarrow & \mathcal{N}(\mathfrak{p}_\mathbb{C})/{K_\mathbb{C}}
\end{array}
\end{align*}

Therefore, we have next two corollaries to 
Theorems \ref{thm:real_minimal} and \ref{thm:number_of_G-orbits_in_minimal}:

\begin{cor}\label{cor:minimal_in_p_C}
For a non-compact real simple Lie algebra $\mathfrak{g}$ without complex structure and its Cartan decomposition $\mathfrak{g} = \mathfrak{k} + \mathfrak{p}$, 
the following conditions on a nilpotent $K_\mathbb{C}$-orbit $\Orbit^{K_\mathbb{C}}$ in $\mathfrak{p}_\C$ are equivalent$:$
	\begin{enumerate}
	\item $\Orbit^{K_\mathbb{C}}$ is minimal in $(\mathcal{N}(\mathfrak{p}_\mathbb{C})/{K_\mathbb{C}}) \setminus \{[0]\}$ with respect to the closure ordering.
	\item The dimension of $\Orbit^{K_\mathbb{C}}$ attains its minimum in $(\mathcal{N}(\mathfrak{p}_\mathbb{C})/{K_\mathbb{C}}) \setminus \{[0]\}$.
	\item $\Orbit^{K_\mathbb{C}}$ is contained in $\ming$.
	\end{enumerate}
\end{cor}

We say that a nilpotent $K_\mathbb{C}$-orbit $\Orbit^{K_\mathbb{C}}$ is \emph{minimal} if $\Orbit^{K_\mathbb{C}}$ satisfies the equivalent conditions on Corollary \ref{cor:minimal_in_p_C}.

\begin{cor}\label{cor:number_min_in_p_C}
\begin{multline*}
\sharp \{\text{ minimal nilpotent $K_\mathbb{C}$-orbits in $\mathfrak{p}_\mathbb{C}$ }\} \\
= 
\begin{cases} 
1 \quad \text{if $(\mathfrak{g},\mathfrak{k})$ is of non-Hermitian type,} \\ 
2 \quad \text{if $(\mathfrak{g},\mathfrak{k})$ is of Hermitian type.}
\end{cases}
\end{multline*}
\end{cor}

\begin{rem}
\begin{itemize}
\item A part of our main results, e.g.~Theorem $\ref{thm:WDD_of_our_orbits}$ \eqref{item:Table} and Theorem $\ref{thm:number_of_G-orbits_in_minimal}$, 
could be proved by using the classification of real nilpotent orbits in $\mathfrak{g}$
$($see Remark $\ref{rem:E6_Djokovic}$ for more details$)$.
In this paper, our proof does not rely on the classification of real nilpotent orbits.
\item Theorem $\ref{thm:number_of_G-orbits_in_minimal}$ 
and Corollary $\ref{cor:number_min_in_p_C}$ should be known to experts.
In particular, 
the claim of \cite[Proposition 2.2]{Kobayashi-Oshima2013-gK} includes 
Corollary $\ref{cor:number_min_in_p_C}$.
For the sake of completeness, we give a proof of Theorem $\ref{thm:number_of_G-orbits_in_minimal}$ in this paper.
\end{itemize}
\end{rem}

The paper is organized as follows.
In Section \ref{section:preliminary},
we recall the definition of weighted Dynkin diagrams of complex nilpotent orbits in complex semisimple Lie algebras 
and some well-known facts for a highest root of a restricted root system of $(\mathfrak{g},\mathfrak{a})$.
We prove Theorems \ref{thm:O_min_g} and \ref{thm:real_minimal} in Section \ref{section:proof_of_MainThm}.
In Section \ref{section:nilp_and_real}, 
we give a proof of the first claim of Theorem \ref{thm:WDD_of_our_orbits}.
We determine the weighted Dynkin diagrams of $\ming$ in Section \ref{section:WDD_of_min}.
Finally, we give a proof of Theorem \ref{thm:number_of_G-orbits_in_minimal} in Section \ref{section:real-orbits_in_min}.

\section{Preliminary results}\label{section:preliminary}

\subsection{Weighted Dynkin diagrams of complex nilpotent orbits}\label{subsection:Hyperbolic_orbits}\label{subsection:WDD_of_nilp}

Let $\mathfrak{g}_\mathbb{C}$ be a complex semisimple Lie algebra.
In this subsection, we recall the definition of weighted Dynkin diagrams of complex nilpotent orbits in $\mathfrak{g}_\mathbb{C}$.

Let us fix a Cartan subalgebra $\mathfrak{h}_\mathbb{C}$ of $\mathfrak{g}_\mathbb{C}$. 
We denote by $\Delta(\mathfrak{g}_\mathbb{C},\mathfrak{h}_\mathbb{C})$ the root system for $(\mathfrak{g}_\mathbb{C},\mathfrak{h}_\mathbb{C})$. 
Then the root system $\Delta(\mathfrak{g}_\mathbb{C},\mathfrak{h}_\mathbb{C})$ becomes a subset of the dual space $\mathfrak{h}^*$ of 
\[
\mathfrak{h} := 
\{\, H \in \mathfrak{h}_\mathbb{C} \mid \alpha (H) \in \mathbb{R} \ \text{for any } \alpha \in \Delta(\mathfrak{g}_\mathbb{C},\mathfrak{h}_\mathbb{C}) \,\}.
\]
We write $W(\mathfrak{g}_\mathbb{C},\mathfrak{h}_\mathbb{C})$ for the Weyl group of $\Delta(\mathfrak{g}_\mathbb{C},\mathfrak{h}_\mathbb{C})$ acting on $\mathfrak{h}$.
Take a positive system $\Delta^+(\mathfrak{g}_\mathbb{C},\mathfrak{h}_\mathbb{C})$ of the root system $\Delta(\mathfrak{g}_\mathbb{C},\mathfrak{h}_\mathbb{C})$.
Then a closed Weyl chamber 
\[
\mathfrak{h}_+ := \{\, H \in \mathfrak{h} \mid \alpha (H) \geq 0 \ \text{for any } \alpha \in \Delta^+(\mathfrak{g}_\mathbb{C},\mathfrak{h}_\mathbb{C}) \,\}
\]
becomes a fundamental domain of $\mathfrak{h}$ for the action of $W(\mathfrak{g}_\mathbb{C},\mathfrak{h}_\mathbb{C})$.

Let $\Pi$ be the simple system of $\Delta^+(\mathfrak{g}_\mathbb{C},\mathfrak{h}_\mathbb{C})$.
Then for each $H \in \mathfrak{h}$, we define a map by
\[
\Psi_H : \Pi \rightarrow \mathbb{R},\ \alpha \mapsto \alpha(H).
\]
We call $\Psi_H$ the weighted Dynkin diagram corresponding to $H \in \mathfrak{h}$, and $\alpha(H)$ the weight on a node $\alpha \in \Pi$ of the weighted Dynkin diagram.
Since $\Pi$ is a basis of $\mathfrak{h}^*$, 
the map
\[
\Psi: \mathfrak{h} \rightarrow \Map(\Pi,\mathbb{R}),\ H \mapsto \Psi_H
\]
is bijective.
Furthermore, 
\[
\mathfrak{h}_+ \rightarrow \Map(\Pi,\mathbb{R}_{\geq 0}),\ H \mapsto \Psi_H
\]
is also bijective.

A triple $(H,X,Y)$ is said to be an $\mathfrak{sl}_2$-triple in $\mathfrak{g}_\mathbb{C}$ if 
\[
[H,X] = 2X,\ [H,Y] = -2Y,\ [X,Y] = H \quad (H,X,Y \in \mathfrak{g}_\mathbb{C}).
\]
For any $\mathfrak{sl}_2$-triple $(H,X,Y)$ in $\mathfrak{g}_\mathbb{C}$, 
the elements $X$ and $Y$ are nilpotent in $\mathfrak{g}_\mathbb{C}$, 
and $H$ is hyperbolic in $\mathfrak{g}_\mathbb{C}$, 
i.e. $\ad_{\mathfrak{g}_\mathbb{C}}H \in \End(\mathfrak{g}_\mathbb{C})$ 
is diagonalizable with only real eigenvalues.

Combining the Jacobson--Morozov theorem with 
the results of Kostant \cite{Kostant59}, 
for each complex nilpotent orbit $\Orbit^{G_\mathbb{C}}$, 
there uniquely exists an element $H_{\Orbit}$ of $\mathfrak{h}_+$ with the following property: 
There exists $X,Y \in \Orbit^{G_\mathbb{C}}$ such that $( H_{\Orbit},X,Y )$ is an $\mathfrak{sl}_2$-triple in $\mathfrak{g}_\mathbb{C}$.
Furthermore, by the results of Malcev \cite{Malcev50}, the following map is injective: 
\[
\mathcal{N}/{G_\C}
\hookrightarrow 
\mathfrak{h}_+,\ \Orbit^{G_\mathbb{C}} \mapsto H_{\Orbit},
\]
where $\mathcal{N}/{G_\C}$ denotes the set of all complex nilpotent orbits in $\mathfrak{g}_\C$.
For each complex nilpotent orbit $\Orbit^{G_\mathbb{C}}$,
the weighted Dynkin diagram corresponding to $H_{\Orbit}$ is called 
the weighted Dynkin diagram of $\Orbit^{G_\mathbb{C}}$.
Dynkin \cite{Dynkin52eng} classified all such weighted Dynkin diagrams for each complex simple Lie algebra $\mathfrak{g}_\mathbb{C}$
as a classification of 
three dimensional simple subalgebras of $\mathfrak{g}_\C$
(see also \cite{Bala-Carter76} for more details).
In particular, by his results, any weight of the weighted Dynkin diagram of $\Orbit^{G_\mathbb{C}}$ is given by $0$, $1$ or $2$ for any complex nilpotent orbit $\Orbit^{G_\mathbb{C}}$.

In the rest of this subsection, we suppose that $\mathfrak{g}_\mathbb{C}$ is simple.
Let $\phi$ be the highest root of $\Delta^+(\mathfrak{g}_\mathbb{C},\mathfrak{h}_\mathbb{C})$.
Then the minimal nilpotent orbit in $\mathfrak{g}_\mathbb{C}$ can be written by 
\[
\minC = G_\mathbb{C} \cdot ((\mathfrak{g}_\mathbb{C})_\phi \setminus \{0\}),
\]
where $(\mathfrak{g}_\mathbb{C})_\phi$ is the root space of $\phi$ in $\mathfrak{g}_\mathbb{C}$.
We denote the coroot of $\phi$ by $H_{\phi}$.
That is, $H_\phi$ is the unique element in $\mathfrak{h}$ with 
\[
\alpha(H_{\phi}) = \frac{2 \langle \alpha, \phi \rangle}{\langle \phi,\phi \rangle} \quad \text{for any } \alpha \in \mathfrak{h}^*,
\]
where $\langle\ ,\ \rangle$ is the inner product on $\mathfrak{h}^*$ induced by the Killing form $B_\C$ on $\mathfrak{g}_\mathbb{C}$.
Since $\phi$ is dominant, $H_{\phi}$ is in $\mathfrak{h}_+$.
Furthermore, $H_{\phi}$ is the hyperbolic element corresponding to $\minC$ since we can find $X_\phi \in \mathfrak{g}_\phi$, $Y_\phi \in \mathfrak{g}_{-\phi}$ such that $(H_{\phi}, X_\phi,Y_\phi)$ is an $\mathfrak{sl}_2$-triple.
Therefore, the weighted Dynkin diagram of $\minC$ is 
\begin{align}
\Psi_{H_{\phi}} : \Pi \rightarrow \mathbb{R}_{\geq 0},\quad \alpha \mapsto \frac{2 \langle \alpha, \phi \rangle}{\langle \phi,\phi \rangle}. \label{def:smallest_and_extended}
\end{align}
In particular, for the cases where $\rank \mathfrak{g}_\mathbb{C} \geq 2$ 
i.e.~$\mathfrak{g}_\mathbb{C}$ is not isomorphic to $\mathfrak{sl}(2,\mathbb{C})$, 
we observe that the weight on $\alpha$ of the weighted Dynkin diagram of $\minC$ is $1$ [resp. $0$] if and only if the nodes $\alpha$ and $-\phi$ are connected [resp. disconnected] by some edges in the extended Dynkin diagram of $\mathfrak{g}_\mathbb{C}$.
The weighted Dynkin diagram of $\minC$ for each simple $\mathfrak{g}_\mathbb{C}$ can be found 
in \cite[Chapter 5.4 and 8.4]{Collingwood-McGovern93} 
(see also Table \ref{table:Satake_diagrams} in Section \ref{subsection:Satake_diagrams}).

\subsection{Highest roots of restricted root systems}\label{subsection:properties_of_restricted_roots}

In this subsection we recall some well-known facts, which will be used for proofs of Theorems \ref{thm:O_min_g} and \ref{thm:real_minimal}, for a highest root of a restricted root system of real semisimple Lie algebra 
without proof. 

Let $\mathfrak{g}_\mathbb{C}$ be a complex simple Lie algebra and $\mathfrak{g}$ a non-compact real form of $\mathfrak{g}$ with a Cartan decomposition $\mathfrak{g} = \mathfrak{k} + \mathfrak{p}$.
We fix a maximal abelian subspace $\mathfrak{a}$ of $\mathfrak{p}$, which is called a \emph{maximally split abelian subspace} of $\mathfrak{g}$, 
and write $\Sigma(\mathfrak{g},\mathfrak{a})$ 
for the restricted root system for $(\mathfrak{g},\mathfrak{a})$.
For each restricted root $\xi$ of $\Sigma(\mathfrak{g},\mathfrak{a})$,
we denote by $A_{\xi} \in \mathfrak{a}$ the coroot of $\xi$.

Then the lemma below holds:

\begin{lem}\label{lem:sl_2_in_g}
For any restricted root $\xi$ of $\Sigma(\mathfrak{g},\mathfrak{a})$ and any non-zero root vector $X_\xi$ in $\mathfrak{g}_\xi$,
there exists $Y_\xi \in \mathfrak{g}_{-\xi}$ such that 
$( A_{\xi}, X_\xi,Y_\xi )$ is an $\mathfrak{sl}_2$-triple in $\mathfrak{g}$.
\end{lem}

We fix an ordering on $\mathfrak{a}$ and write $\Sigma^+(\mathfrak{g},\mathfrak{a})$ for the positive system of $\Sigma(\mathfrak{g},\mathfrak{a})$ corresponding to the ordering on $\mathfrak{a}$.
We denote by $\lambda$ the highest root of $\Sigma^+(\mathfrak{g},\mathfrak{a})$ with respect to the ordering on $\mathfrak{a}$.
Next lemma claims that the highest root $\lambda$ depends only on the positive system $\Sigma^+(\mathfrak{g},\mathfrak{a})$ but not on the ordering on $\mathfrak{a}$:

\begin{lem}\label{lem:restricted_highest_root}
The highest root $\lambda$ of $\Sigma^+(\mathfrak{g},\mathfrak{a})$ is 
the unique dominant longest root of $\Sigma(\mathfrak{g},\mathfrak{a})$.
\end{lem}

The following lemma gives a characterization of the highest root $\lambda$ of $\Sigma^+(\mathfrak{g},\mathfrak{a})$:

\begin{lem}\label{lem:Ker_of_n}
Let $\xi$ be a root of $\Sigma(\mathfrak{g},\mathfrak{a})$.
If $\xi$ is not the highest root, 
then for any non-zero root vector $X_\xi$ in $\mathfrak{g}_\xi$, 
there exists a positive root $\eta$ in $\Sigma^+(\mathfrak{g},\mathfrak{a})$ and 
a root vector $X_\eta \in \mathfrak{g}_\eta$ such that $[X_\xi,X_\eta] \neq 0$.
In particular, $\xi$ is the highest root 
if and only if 
$\xi + \eta \in \mathfrak{a}^*$ is not a root of $\Sigma(\mathfrak{g},\mathfrak{a})$ 
for any $\eta \in \Sigma^+(\mathfrak{g},\mathfrak{a})$. 
\end{lem}

\section{Proofs of Theorem \ref{thm:O_min_g} and Theorem \ref{thm:real_minimal}}\label{section:proof_of_MainThm}

We consider the same setting in Section \ref{subsection:properties_of_restricted_roots},
and fix connected Lie groups $G_\mathbb{C}$ and $G$ with its Lie algebras 
$\mathfrak{g}_\mathbb{C}$ and $\mathfrak{g}$, respectively. 

In this section, we give proofs of 
Theorems \ref{thm:O_min_g} and \ref{thm:real_minimal}.
To this, we prove the next two lemmas: 

\begin{lem}\label{lem:highest_in_closure}
Let $\Orbit'_0$ be a non-zero real nilpotent orbit in $\mathfrak{g}$.
Then there exists a non-zero highest root vector $X_\lambda$ in $\mathfrak{g}_{\lambda}$ such that $X_\lambda$ is contained in the closure of $\Orbit'_0$.
\end{lem}

\begin{lem}\label{lem:G_C-conj_of_highest_root_vectors}
For any two non-zero highest root vectors $X_\lambda$, $X_\lambda'$ in $\mathfrak{g}_\lambda$,
there exists $g \in G_\mathbb{C}$ such that $g X_\lambda =X_\lambda'$.
\end{lem}

Theorems \ref{thm:O_min_g} and \ref{thm:real_minimal} follows from 
Lemmas \ref{lem:highest_in_closure} and \ref{lem:G_C-conj_of_highest_root_vectors} immediately.

We also remark that 
Lemma \ref{lem:highest_in_closure} implies the next proposition, 
which will be used in Section \ref{section:real-orbits_in_min} to prove Theorem \ref{thm:number_of_G-orbits_in_minimal}.

\begin{prop}\label{prop:G-orbits_meets_highest}
Any $G$-orbit in $\ming \cap \mathfrak{g}$ meets $\mathfrak{g}_\lambda \setminus \{0\}$.
\end{prop}

Let us give proofs of Lemmas \ref{lem:highest_in_closure} and \ref{lem:G_C-conj_of_highest_root_vectors} as follows.

\begin{proof}[Proof of Lemma $\ref{lem:highest_in_closure}$]
There is no loss of generality in assuming that the ordering on $\mathfrak{a}$ is lexicographic.
Let us put $\mathfrak{m} = Z_{\mathfrak{k}}(\mathfrak{a})$.
Then $\mathfrak{g}$ can be decomposed as
\[
\mathfrak{g} = \mathfrak{m} \oplus \mathfrak{a} \oplus \bigoplus_{\xi \in \Sigma(\mathfrak{g},\mathfrak{a})} \mathfrak{g}_{\xi}.
\]
For each $X' \in \mathfrak{g}$, we denote by
\[
X' = X'_\mathfrak{m} + X'_\mathfrak{a} + \sum_{\xi \in \Sigma(\mathfrak{g},\mathfrak{a})} X'_\xi \quad (X'_\mathfrak{m} \in \mathfrak{m},\ X'_\mathfrak{a} \in \mathfrak{a},\ X'_\xi \in \mathfrak{g}_\xi).
\]
For a fixed $X' \in \overline{\Orbit'_0}$, we denote by $\lambda'$ the highest root of 
\[
\Sigma_{X'} := \{\, \xi \in \Sigma(\mathfrak{g},\mathfrak{a}) \mid X'_\xi \neq 0 \,\}
\]
with respect to the ordering on $\mathfrak{a}$.
Here we remark that if $X' \neq 0$, then 
the set $\Sigma_{X'}$ is not empty
since $X'$ is nilpotent element in $\mathfrak{g}$.
As a first step of the proof, we shall prove that for any $X' \in \overline{\Orbit'_0}$, the root vector $X'_{\lambda'}$ is also in $\overline{\Orbit'_0}$.
Let us take $A' \in \mathfrak{a}$ satisfying that 
\begin{align*}
0 < \lambda'(A') \text{ and }
\xi(A') < \lambda'(A') \quad \text{for any } \xi \in \Sigma_{X'} \setminus \{\lambda'\}.
\end{align*}
Note that such $A'$ exists 
since $\lambda'$ is the highest root of $\Sigma_{X'}$ 
with respect to the lexicographic ordering on $\mathfrak{a}$.
Let us put 
\[
X'_k := \frac{1}{e^{k\lambda'(A')}} \exp(\ad_\mathfrak{g}k A') X' \quad \text{for each } k \in \mathbb{N}.
\]
Then $X'_k$ is in $\overline{\Orbit'_0}$ for each $k$ since $\overline{\Orbit'_0}$ is stable by positive scalars.
Furthermore, since
\begin{align*}
\lim_{k \rightarrow \infty} X'_k 
	&= \lim_{k \rightarrow \infty} \sum_{\xi \in \Sigma_{X'}} e^{k(\xi(A')-\lambda'(A'))} X'_\xi = X'_{\lambda'},
\end{align*}
we obtain that $X'_{\lambda'}$ is in $\overline{O'_0}$.
To complete the proof, 
we only need to show that there exists $X' \in \overline{O'_0}$ such that 
$\lambda' = \lambda$, where $\lambda'$ is the highest root of $\Sigma_{X'}$.
Let us put 
\begin{align*}
\Sigma_{\overline{O}'_0} 
&:= \{\, \xi \in \Sigma(\mathfrak{g},\mathfrak{a}) \mid \ \text{there exists } X' \in \overline{\Orbit'_0} \text{ such that } X'_\xi \neq 0 \,\} \\
&= \bigcup_{X' \in \overline{\Orbit'_0}} \Sigma_{X'}.
\end{align*}
We denote by $\lambda_0$ the highest root of $\Sigma_{\overline{O}'_0}$.
Then we can find a root vector $X'_{\lambda_0}$ in $\mathfrak{g}_{\lambda_0} \cap \overline{\Orbit'_0}$ by using the fact proved above.
We assume that $\lambda_0 \neq \lambda$.
Then by Lemma \ref{lem:Ker_of_n}, we can find $\eta \in \Sigma^+(\mathfrak{g},\mathfrak{a})$ and 
$X_\eta \in \mathfrak{g}_\eta$ such that $[X'_{\lambda_0},X_\eta] \neq 0$.
Thus we have 
\[
\lambda_0 + \eta \in \Sigma_{X''} \subset \Sigma_{\overline{\Orbit'_0}}.
\]
where $X'' := \exp(\ad_{\mathfrak{g}} X_\eta) X'_{\lambda_0} \in \overline{\Orbit'_0}$.
This contradicts the definition of $\lambda_0$.
Thus $\lambda_0 = \lambda$.
\end{proof}

\begin{proof}[Proof of Lemma $\ref{lem:G_C-conj_of_highest_root_vectors}$]
Fix non-zero highest root vectors $X_\lambda$ and $X'_\lambda$.
Let $A_{\lambda}$ the coroot of $\lambda$ in $\mathfrak{a}$.
Then by Lemma \ref{lem:sl_2_in_g}, 
we can find $Y_{\lambda}$ and $Y_{\lambda}'$ in $\mathfrak{g}_\mathbb{C}$ such that 
$(A_{\lambda},X_{\lambda},Y_{\lambda})$
and
$(A_{\lambda},X'_{\lambda},Y'_{\lambda})$ 
are $\mathfrak{sl}_2$-triples in $\mathfrak{g}_\mathbb{C}$, respectively.
Thus by applying Malcev's theorem in \cite{Malcev50} 
there exists $g \in G_\mathbb{C}$ such that $g X_\lambda = X'_\lambda$.
\end{proof}

\section{Complex nilpotent orbits and real forms}\label{section:nilp_and_real}

Let $\mathfrak{g}_\mathbb{C}$ be a complex simple Lie algebra and $\mathfrak{g}$ a non-compact real form of $\mathfrak{g}_\mathbb{C}$. 
In this section, 
we will give a necessary and sufficient condition of $\mathfrak{g}$ for 
$\minC = \ming$ including the first claim of Theorem \ref{thm:WDD_of_our_orbits}.

We fix $G$, $G_\mathbb{C}$ for the connected Lie group with its Lie algebra $\mathfrak{g}$, $\mathfrak{g}_\mathbb{C}$, respectively.
Let $\mathfrak{g} = \mathfrak{k} + \mathfrak{p}$ be a Cartan decomposition of $\mathfrak{g}$. 
We fix a maximal abelian subspace $\mathfrak{a}$ of $\mathfrak{p}$ and its ordering.
Let $\lambda$ be the highest root of the restricted root system $\Sigma(\mathfrak{g},\mathfrak{a})$ for $(\mathfrak{g},\mathfrak{a})$ with respect to the ordering on $\mathfrak{a}$.
Then by Theorem \ref{thm:O_min_g}, which was already proved in Section \ref{section:proof_of_MainThm}, the complex nilpotent orbit \[
\ming = G_\mathbb{C} \cdot (\mathfrak{g}_\lambda \setminus \{0\})
\]
is the minimum in $(\mathcal{N}_{\mathfrak{g}}/G_\mathbb{C}) \setminus \{[0]\}$.

We extend $\mathfrak{a}$ and its ordering to a Cartan subalgebra $\mathfrak{h}_\mathbb{C}$ of $\mathfrak{g}_\mathbb{C}$ and an ordering on it.
Let $\phi$ be the highest root of the root system $\Delta(\mathfrak{g}_\mathbb{C},\mathfrak{h}_\mathbb{C})$ for $(\mathfrak{g}_\mathbb{C},\mathfrak{h}_\mathbb{C})$ with respect to the ordering on $\mathfrak{h}_\mathbb{C}$.
We recall that the complex nilpotent orbit \[
\minC = G_\mathbb{C} \cdot ((\mathfrak{g}_\mathbb{C})_\phi \setminus \{0\})
\]
is the minimum in $(\mathcal{N}/G_\mathbb{C}) \setminus \{[0]\}$.

Then the next proposition, 
including the first claim of Theorem \ref{thm:WDD_of_our_orbits}, 
holds:

\begin{prop}\label{prop:min_min}
The following conditions on a non-compact real simple Lie algebra $\mathfrak{g}$ without complex structure are equivalent$:$
\begin{enumerate}
\item \label{item:min_min:min_min} $\minC \neq \ming$.
\item \label{item:min_min:min_meets_g} $\minC \cap \mathfrak{g} = \emptyset$. 
\item \label{item:min_min:min_meets_p} $\minC \cap \mathfrak{p}_\mathbb{C} = \emptyset$ 
where $\mathfrak{g}_\mathbb{C} = \mathfrak{k}_\mathbb{C} + \mathfrak{p}_\mathbb{C}$ is the complexification of a Cartan decomposition of $\mathfrak{g}$.
\item \label{item:min_min:one_dim_highest} $\dim_\mathbb{R} \mathfrak{g}_\lambda \geq 2$. 
\item \label{item:min_min:real} The highest root $\phi$ of $\Delta(\mathfrak{g}_\mathbb{C},\mathfrak{h}_\mathbb{C})$ defined above is not a real root.
\item \label{item:min_min:WDD_match} The weighted Dynkin diagram of $\minC$ does not match the Satake diagram of $\mathfrak{g}$ $($see Section $\ref{subsection:Satake_diagrams}$ for the notation$)$.
\item \label{item:min_min:WDD_extended} 
There exists a node $\alpha$ of Dynkin diagram of $\mathfrak{g}_\mathbb{C}$ such that $\alpha$ is black in the Satake diagram of $\mathfrak{g}$ and has some edges connected to the added node in the extended Dynkin diagram of $\mathfrak{g}_\mathbb{C}$.
\item \label{item:min_min:extra:discrete} There exists an infinite-dimensional 
$($non-holomorphic$)$ irreducible $(\mathfrak{g}_\C,G_U)$-module $X$ 
such that $X$ is discretely decomposable as a $(\mathfrak{g},K)$-module,
where $G_U$ is a connected compact real form of $G_\C$
$($See \cite[Section 1.2]{Kobayashi97discreteIII} for the definition of 
the discrete decomposability$)$.
\item \label{item:min_min:extra:discreteK} There exists an infinite-dimensional 
$($non-holomorphic$)$ irreducible $(\mathfrak{g}_\C,G_U)$-module $X$ 
such that $X$ is discretely decomposable as a $(\mathfrak{k}_\C,K)$-module,
where $G_U$ is a connected compact real form of $G_\C$.
\item \label{item:min_min:extra:nilp_projection}
$\mathrm{pr}_{\mathfrak{p}_\C}(\Orbit_{\min}^{G_\mathbb{C}})$ is contained in the nilpotent cone $\mathcal{N}(\mathfrak{p}_\C) := \mathcal{N} \cap \mathfrak{p}_\C$ in $\mathfrak{p}_\C$,
where $\mathrm{pr}_{\mathfrak{p}_\C} : \mathfrak{g}_\mathbb{C} \rightarrow \mathfrak{p}_\mathbb{C}$ denotes the projection with respect to the decomposition $\mathfrak{g}_\mathbb{C} = \mathfrak{k}_\mathbb{C} + \mathfrak{p}_\mathbb{C}$. 
\item \label{item:min_min:extra:nilp_projectionK}
$\mathrm{pr}_{\mathfrak{k}_\C}(\Orbit_{\min}^{G_\mathbb{C}})$ is contained in the nilpotent cone $\mathcal{N} \cap \mathfrak{k}_\C$ in $\mathfrak{k}_\C$,
where $\mathrm{pr}_{\mathfrak{k}_\C} : \mathfrak{g}_\mathbb{C} \rightarrow \mathfrak{k}_\mathbb{C}$ denotes the projection with respect to the decomposition $\mathfrak{g}_\mathbb{C} = \mathfrak{k}_\mathbb{C} + \mathfrak{p}_\mathbb{C}$. 
\item \label{item:min_min:classification} 
$\mathfrak{g}$ is isomorphic to one of the following simple Lie algebras 
\begin{align*}
\mathfrak{su}^*(2k),\ \mathfrak{so}(n,1),\ \mathfrak{sp}(p,q),\ \mathfrak{e}_{6(-26)} \text{ and } \mathfrak{f}_{4(-20)},
\end{align*}
for $k \geq 2$, $n \geq 5$ and $p,q \geq 1$.
\end{enumerate}
\end{prop}

The equivalences among 
\eqref{item:min_min:min_min}, 
\eqref{item:min_min:min_meets_g}
and 
\eqref{item:min_min:min_meets_p}
follow from the definition of $\ming$ and \cite[Proposition 1.11]{Sekiguchi87}.

By the list of the Satake diagrams of non-compact real simple Lie algebras (see also Table \ref{table:Satake_diagrams} for the Satake diagram of each $\mathfrak{g}$)
 and the extended Dynkin diagrams of complex simple Lie algebras,
one can easily check the equivalence 
\eqref{item:min_min:WDD_extended} 
$\Leftrightarrow$
\eqref{item:min_min:classification}.

The equivalences among
\eqref{item:min_min:real}, 
\eqref{item:min_min:extra:discrete},
\eqref{item:min_min:extra:nilp_projection},
\eqref{item:min_min:extra:nilp_projectionK}
and 
\eqref{item:min_min:classification}
were proved by \cite[Lemma 4.6 and Theorem 5.2]{Kobayashi-Oshima2013-gK}
in a more general setting.
In particular, the equivalences 
\eqref{item:min_min:real} 
$\Leftrightarrow$ 
\eqref{item:min_min:extra:discrete}
$\Leftrightarrow$ 
\eqref{item:min_min:extra:nilp_projection} 
and 
\eqref{item:min_min:real} 
$\Leftrightarrow$
\eqref{item:min_min:extra:discreteK} 
$\Leftrightarrow$
\eqref{item:min_min:extra:nilp_projectionK}
can be obtained by applying their results
for the symmetric pairs 
$(\mathfrak{g}_\C,\mathfrak{g})$ and 
$(\mathfrak{g}_\C,\mathfrak{k}_\C)$,
respectively
(see also \cite[Remark 4.5]{Kobayashi-Oshima2013-gK} 
for the discrete decomposability of a representation of $G$ 
with respect to a symmetric pair $(G,H)$ and its associated pair $(G,H^a)$).

In this section,
we give a proof of the remaining equivalences,
namely,
the equivalences among 
\eqref{item:min_min:min_meets_g},
\eqref{item:min_min:one_dim_highest},
\eqref{item:min_min:real},
\eqref{item:min_min:WDD_match}
and 
\eqref{item:min_min:WDD_extended}.

Note that 
the equivalence 
\eqref{item:min_min:min_meets_g} 
$\Leftrightarrow$
\eqref{item:min_min:real},
which will be proved in Section \ref{subsection:proofProp} of this paper,
is used in a proof of 
\cite[Corollary 5.9]{Kobayashi-Oshima2013-gK}.

\begin{rem}
The equivalences among
\eqref{item:min_min:min_meets_g},
\eqref{item:min_min:extra:nilp_projectionK}
and
\eqref{item:min_min:classification} in Proposition $\ref{prop:min_min}$
were stated on Brylinski's paper \cite{Brylinski98} without proof.
It should be noted that Brylinski \cite{Brylinski98} also claimed that 
the following condition on $\mathfrak{g}$ is also equivalent to the condition \eqref{item:min_min:min_meets_g}$:$
\begin{itemize}
\item $K_\mathbb{C}$ has a Zariski open orbit in $\Orbit_{\min}^{G_\mathbb{C}}$, where $K_\mathbb{C}$ is the adjoint group of $\mathfrak{k}_\mathbb{C}$.
\end{itemize}
\end{rem}

\subsection{Satake diagrams and weighted Dynkin diagrams of complex nilpotent orbits}\label{subsection:Satake_diagrams}

In order to explain the notation in \eqref{item:min_min:WDD_match},
we first recall the definition of the Satake diagram of a real form $\mathfrak{g}$ of $\mathfrak{g}_\mathbb{C}$ briefly.
All facts which will be used for the definition of the Satake diagrams can be found in \cite{Araki62} or \cite{Satake60}.
Throughout this subsection, 
$\mathfrak{g}_\mathbb{C}$ can be a general complex semisimple Lie algebra and 
$\mathfrak{g}$ a general real form of $\mathfrak{g}_\mathbb{C}$.

We fix a Cartan decomposition $\mathfrak{g} = \mathfrak{k} + \mathfrak{p}$ of $\mathfrak{g}$. 
Take a maximal abelian subspace $\mathfrak{a}$ in $\mathfrak{p}$, and extend it to a maximal abelian subspace $\mathfrak{h} = \sqrt{-1}\mathfrak{t} + \mathfrak{a}$ in $\sqrt{-1}\mathfrak{k} + \mathfrak{p}$.
Then the complexification, denoted by $\mathfrak{h}_\mathbb{C}$, of $\mathfrak{h}$ is a Cartan subalgebra of $\mathfrak{g}_\mathbb{C}$, and $\mathfrak{h}$ coincides with the real form 
\[
\{ X \in \mathfrak{h}_\mathbb{C} \mid \alpha(X) \in \mathbb{R} \ \text{for any } \alpha \in \Delta(\mathfrak{g}_\mathbb{C},\mathfrak{h}_\mathbb{C})\}
\]
of $\mathfrak{h}_\mathbb{C}$ where $\Delta(\mathfrak{g}_\mathbb{C},\mathfrak{h}_\mathbb{C})$ is the reduced root system for $(\mathfrak{g}_\mathbb{C},\mathfrak{h}_\mathbb{C})$. 
Let us denote by \[
\Sigma(\mathfrak{g},\mathfrak{a}) := \{ \alpha|_\mathfrak{a} \mid \alpha \in \Delta(\mathfrak{g}_\mathbb{C},\mathfrak{h}_\mathbb{C}) \} \setminus \{0\} \subset \mathfrak{a}^*
\] the restricted root system for $(\mathfrak{g},\mathfrak{a})$.
We will denote by $W(\mathfrak{g},\mathfrak{a})$, $W(\mathfrak{g}_\mathbb{C},\mathfrak{h}_\mathbb{C})$ the Weyl group of $\Sigma(\mathfrak{g},\mathfrak{a})$, $\Delta(\mathfrak{g}_\mathbb{C},\mathfrak{h}_\mathbb{C})$, respectively.
Fix an ordering on $\mathfrak{a}$ and extend it to an ordering on $\mathfrak{h}$.
We write $\Sigma^+(\mathfrak{g},\mathfrak{a})$, $\Delta^+(\mathfrak{g}_\mathbb{C},\mathfrak{h}_\mathbb{C})$ for the positive system of $\Sigma(\mathfrak{g},\mathfrak{a})$, $\Delta(\mathfrak{g}_\mathbb{C},\mathfrak{h}_\mathbb{C})$ corresponding to the ordering on $\mathfrak{a}$, $\mathfrak{h}$, respectively.
Then $\Sigma^+(\mathfrak{g},\mathfrak{a})$ can be written by 
\[
\Sigma^+(\mathfrak{g},\mathfrak{a}) = 
\{ \alpha|_\mathfrak{a} \mid \alpha \in \Delta^+(\mathfrak{g}_\mathbb{C},\mathfrak{h}_\mathbb{C}) \} \setminus \{0\}.
\]
We denote by $\Pi$ the fundamental system of $\Delta^+(\mathfrak{g}_\mathbb{C},\mathfrak{h}_\mathbb{C})$.
Then \[
\overline{\Pi} = \{\, \alpha|_\mathfrak{a} \mid \alpha \in \Pi \,\} \setminus \{0\}
\] becomes the simple system of $\Sigma^+(\mathfrak{g},\mathfrak{a})$.
Let us denote by $\Pi_0$ the set of all simple roots in $\Pi$ whose restrictions to $\mathfrak{a}$ are zero.

The Satake diagram $S$ of $\mathfrak{g}$ consists of the following three data: 
the Dynkin diagram of $\mathfrak{g}_\mathbb{C}$ with nodes $\Pi$, 
black nodes $\Pi_0$ in $S$, 
and arrows joining $\alpha \in \Pi \setminus \Pi_0$ and $\beta \in \Pi \setminus \Pi_0$ in $S$ 
whose restrictions to $\mathfrak{a}$ are the same.

Second, we define the relation ``match'' 
 between an weighted Dynkin diagram and a Satake diagram as follows:

\begin{defn}[{\cite[Definition 7.3]{Okuda13cls}}]\label{defn:match}
Let $\Psi_H \in \Map(\Pi,\mathbb{R})$ be a weighted Dynkin diagram 
$($see Section $\ref{subsection:WDD_of_nilp}$ for the definition$)$ 
and $S$ the Satake diagram of $\mathfrak{g}$ with nodes $\Pi$ defined above.
We say that $\Psi_H$ \textit{matches} $S$ if all the weights on black nodes are zero and any pair of nodes joined by an arrow has the same weights.
\end{defn}

\begin{rem}
The concept of ``match'' appeared earlier in Djokovic \cite{Djokovic82}
$($weighted Satake diagrams$)$ and Sekiguchi \cite[Proposition 1.16]{Sekiguchi84}. 
\end{rem}

The following two facts were proved in \cite{Okuda13cls}.
In particular, by using Fact \ref{thm:nilp_match},
one can easily check whether or not a given complex nilpotent orbit meets a given real form.

\begin{fact}\label{fact:a_to_Satake}
The bijection $\Psi$ between $\mathfrak{h}$ and $\Map(\Pi,\mathbb{R})$ defined in Section $\ref{subsection:Hyperbolic_orbits}$ induces a bijection below$:$
\[
\mathfrak{a} \bijarrow \{\, \Psi_H \in \Map(\Pi,\mathbb{R}) \mid \text{$\Psi_H$ matches $S$} \,\}.
\]
\end{fact}

\begin{fact}[{\cite[Proposition 7.8 and Theorem 7.10]{Okuda13cls}}]\label{thm:nilp_match}
Let $\mathfrak{g}_\mathbb{C}$ be a complex semisimple Lie algebra and $\mathfrak{g}$ a real form of $\mathfrak{g}_\mathbb{C}$.
For a complex nilpotent orbit $\Orbit^{G_\mathbb{C}}$ in $\mathfrak{g}_\mathbb{C}$,
the following conditions are equivalent$:$
\begin{enumerate}
\item The orbit $\Orbit^{G_\mathbb{C}}$ meets $\mathfrak{g}$.
\item The hyperbolic element $H_{\Orbit}$ corresponding to $\Orbit^{G_\mathbb{C}}$ is in $\mathfrak{a}$ $($see Section $\ref{subsection:Hyperbolic_orbits}$ for the notation$)$.
\item The weighted Dynkin diagram of $\Orbit^{G_\mathbb{C}}$ matches the Satake diagram of $\mathfrak{g}$ 
$($see Section $\ref{subsection:Satake_diagrams}$ for the notation$)$.
\end{enumerate}
\end{fact}

We give examples for Fact \ref{thm:nilp_match} as follows:

\begin{example}
If $\mathfrak{g}$ is a split real form of $\mathfrak{g}_\mathbb{C}$,
then all nodes of the Satake diagram of $\mathfrak{g}$ are white with no arrows.
Thus all complex nilpotent orbits in $\mathfrak{g}_\mathbb{C}$ meets $\mathfrak{g}$ since all weighted Dynkin diagram matches the Satake diagram of $\mathfrak{g}$.
\end{example}

\begin{example}
If $\mathfrak{u}$ is a compact real form of $\mathfrak{g}_\mathbb{C}$,
then all nodes of the Satake diagram of $\mathfrak{u}$ are black.
Thus any non-zero complex nilpotent orbit in $\mathfrak{g}_\mathbb{C}$ does not meet $\mathfrak{u}$ since any non-zero weighted Dynkin diagram does not matches the Satake diagram of $\mathfrak{u}$.
\end{example}

By the list of the weighted Dynkin diagrams of the minimal nilpotent orbit $\minC$ for simple $\mathfrak{g}_\mathbb{C}$ (cf. \cite[Chapter 5.4 and 8.4]{Collingwood-McGovern93}) and the list of the Satake diagrams of non-compact real forms $\mathfrak{g}$, one can easily check that $\minC$ meets $\mathfrak{g}$ or not as follows:

\begin{example}\label{example:min_meets_g}
In Table $\ref{table:Satake_diagrams}$, 
we check whether or not the minimal nilpotent orbit $\minC$ in a complex simple Lie algebra $\mathfrak{g}_\mathbb{C}$ meets a non-compact real form $\mathfrak{g}$.
\begin{longtable}{lll} 
\caption{List of the weighted Dynkin diagram of $\minC$ and the Satake diagram of $\mathfrak{g}$.} \\
\hline \hline
	$\mathfrak{g}$ & Weighted Dynkin diagram of $\minC$ & $\minC$ meets $\mathfrak{g}$? \\ 
	& on the Satake diagram of $\mathfrak{g}$ \\ \midrule 
	$\mathfrak{sl}(n,\mathbb{R})$ & \begin{xy}
	*++!D{1} *\cir<2pt>{}        ="A",
	(10,0) *++!D{0} *\cir<2pt>{} ="B",
	(20,0) *++!D{0} *\cir<2pt>{} ="C",
	(35,0) *++!D{0} *\cir<2pt>{} ="D",
	(45,0) *++!D{0} *\cir<2pt>{} ="E",
	(55,0) *++!D{1} *+!U{\alpha_{n-1}} *\cir<2pt>{} ="F",
	\ar@{-} "A";"B"
	\ar@{-} "B";"C"
	\ar@{-} "C"; (25,0)
	\ar@{.} (25,0) ; (30,0)^*!U{\cdots}
	\ar@{-} (30,0) ; "D"
	\ar@{-} "D";"E"
	\ar@{-} "E";"F"
\end{xy} & Yes \\ \hline
	$\mathfrak{su}^*(2k)$ & \begin{xy}
	*++!D{1} *{\bullet}        ="A",
	(10,0) *++!D{0} *\cir<2pt>{} ="B",
	(20,0) *++!D{0} *{\bullet} ="C",
	(35,0) *++!D{0} *{\bullet} ="D",
	(45,0) *++!D{0} *\cir<2pt>{} ="E",
	(55,0) *++!D{1} *+!U{\alpha_{2k-1}} *{\bullet} ="F",
	\ar@{-} "A";"B"
	\ar@{-} "B";"C"
	\ar@{-} "C"; (25,0)
	\ar@{.} (25,0) ; (30,0)^*!U{\cdots}
	\ar@{-} (30,0) ; "D"
	\ar@{-} "D";"E"
	\ar@{-} "E";"F"
\end{xy} & No \\ \hline
	\raisebox{-1cm}{$\mathfrak{su}(n-p,p)$} & \begin{xy}
	*++!D{1} *+!UR{\alpha_1} *\cir<2pt>{} ="A",
	(10,0) *++!D{0} *\cir<2pt>{} ="B",
	(25,0) *++!D{0} *+!UR{\alpha_p} *\cir<2pt>{} ="C",
	(35,0) *++!D{0} *{\bullet} ="D",
	(40,-5) *++!D{0} *{\bullet} ="E",
	(40,-20) *++!DR{0} *{\bullet} ="F",
	(35,-25) *++!DR{0} *{\bullet} ="G",
	(25,-25) *++!DR{0} *+!U{\alpha_{n-p}} *\cir<2pt>{} ="H",
	(10,-25) *++!DR{0} *\cir<2pt>{} ="I",
	(0,-25) *++!DR{1} *+!U{\alpha_{n-1}} *\cir<2pt>{} ="J",
	\ar@{-} "A";"B"
	\ar@{-} "B"; (15,0)
	\ar@{.} (15,0) ; (20,0)^*!U{\cdots}
	\ar@{-} (20,0) ;"C"
	\ar@{-} "C";"D"
	\ar@{-} "D";"E"
    \ar@{-} "E";(40,-10)
    \ar@{.} (40,-10);(40,-15)
    \ar@{-} (40,-15);"F"
	\ar@{-} "G";"F"
	\ar@{-} "H";"G"
	\ar@{-} (20,-25) ;"H"
	\ar@{.} (15,-25) ; (20,-25)^*!U{\cdots}
	\ar@{-} "I"; (15,-25)
	\ar@{-} "J";"I"
	\ar@{<->} (0,-3);(0,-22)
	\ar@{<->} (10,-3);(10,-22)
	\ar@{<->} (25,-3);(25,-22)
\end{xy} & Yes \\ \hline
	\raisebox{-0.5cm}{$\mathfrak{su}(k,k)$} & \begin{xy}
	*++!D{1} *+!UR{\alpha_1} *\cir<2pt>{} ="A",
	(10,0) *++!D{0} *\cir<2pt>{} ="B",
	(25,0) *++!D{0} *+!UR{\alpha_{k-1}} *\cir<2pt>{} ="C",
	(30,-8) *++!D{0} *+!L{\alpha_{k}} *\cir<2pt>{} ="D",
	(25,-16) *++!DR{0} *\cir<2pt>{} ="E",
	(10,-16) *++!DR{0} *\cir<2pt>{} ="F",
	(0,-16) *++!DR{1} *+!U{\alpha_{2k-1}} *\cir<2pt>{} ="G",
	\ar@{-} "A";"B"
	\ar@{-} "B";(15,0) 
	\ar@{.} (15,0);(20,0)^*!U{\cdots}
	\ar@{-} (20,0);"C"
	\ar@{-} "C";"D"
	\ar@{-} "D";"E"
	\ar@{-} "E";(20,-16)
	\ar@{.} (15,-16);(20,-16)^*!U{\cdots}
	\ar@{-} (15,-16);"F"
	\ar@{-} "F";"G"
	\ar@{<->} (0,-2);(0,-14)
	\ar@{<->} (10,-2);(10,-14)
	\ar@{<->} (25,-2);(25,-14)
\end{xy} & Yes \\ \hline	
	$\mathfrak{so}(2k+1-p,p)$ & \begin{xy}
	*++!D{0} *\cir<2pt>{}        ="A",
	(7,0) *++!D{1} *\cir<2pt>{} ="B",
	(14,0) *++!D{0} *\cir<2pt>{} ="C",
	(28,0) *++!D{0} *+!U{\alpha_p} *\cir<2pt>{} ="D",
	(35,0) *++!D{0} *{\bullet} ="E",
	(49,0) *++!D{0} *{\bullet} ="F",	
	(56,0) *++!D{0} *+!U{\alpha_{k}} *{\bullet} ="G",
	\ar@{-} "A";"B"
	\ar@{-} "B";"C"
	\ar@{-} "C"; (19,0)
	\ar@{.} (19,0) ; (23,0)^*!U{\cdots}
	\ar@{-} (23,0) ; "D"
	\ar@{-} "D";"E"
	\ar@{-} "E"; (40,0)
	\ar@{.} (40,0) ; (44,0)^*!U{\cdots}
	\ar@{-} (44,0) ; "F"
	\ar@{=>} "F";"G"
\end{xy} & No if $p = 1$ \\ \hline
	$\mathfrak{so}(k+1,k)$ & \begin{xy}
	*++!D{0} *\cir<2pt>{}        ="A",
	(7,0) *++!D{1} *\cir<2pt>{} ="B",
	(14,0) *++!D{0} *\cir<2pt>{} ="C",
	(49,0) *++!D{0} *\cir<2pt>{} ="D",
	(56,0) *++!D{0} *+!U{\alpha_k} *\cir<2pt>{} ="E",
	\ar@{-} "A";"B"
	\ar@{-} "B";"C"
	\ar@{-} "C"; (19,0)
	\ar@{.} (19,0) ; (44,0)^*!U{\cdots}
	\ar@{-} (44,0) ; "D"
	\ar@{=>} "D";"E"
\end{xy} & Yes \\ \hline
	$\mathfrak{sp}(k,\mathbb{R})$ & \begin{xy}
	*++!D{1} *\cir<2pt>{}        ="A",
	(7,0) *++!D{0} *\cir<2pt>{} ="B",
	(14,0) *++!D{0} *\cir<2pt>{} ="C",
	(49,0) *++!D{0} *\cir<2pt>{} ="D",
	(56,0) *++!D{0} *+!U{\alpha_k} *\cir<2pt>{} ="E",
	\ar@{-} "A";"B"
	\ar@{-} "B";"C"
	\ar@{-} "C"; (19,0)
	\ar@{.} (19,0) ; (44,0)^*!U{\cdots}
	\ar@{-} (44,0) ; "D"
	\ar@{<=} "D";"E"
\end{xy} & Yes \\ \hline
	$\mathfrak{sp}(k-p,p)$ & \begin{xy}
	*++!D{1} *{\bullet}       ="A",
	(7,0) *++!D{0} *\cir<2pt>{} ="B",
	(14,0) *++!D{0} *{\bullet} ="C",
	(28,0) *++!D{0} *+!U{\alpha_{2p}} *\cir<2pt>{} ="D",
	(35,0) *++!D{0} *{\bullet} ="E",
	(49,0) *++!D{0} *{\bullet} ="F",	
	(56,0) *++!D{0} *+!U{\alpha_k} *{\bullet} ="G",
	\ar@{-} "A";"B"
	\ar@{-} "B";"C"	
	\ar@{-} "C"; (19,0)
	\ar@{.} (19,0) ; (23,0)^*!U{\cdots}
	\ar@{-} (23,0) ; "D"
	\ar@{-} "D";"E"
	\ar@{-} "E"; (40,0)
	\ar@{.} (40,0) ; (44,0)^*!U{\cdots}
	\ar@{-} (44,0) ; "F"
	\ar@{<=} "F";"G"
\end{xy} & No \\ \hline
	$\mathfrak{sp}(m,m)$ & \begin{xy}
	*++!D{1} *{\bullet}       ="A",
	(7,0) *++!D{0} *\cir<2pt>{} ="B",
	(14,0) *++!D{0} *{\bullet} ="C",
	(42,0) *++!D{0} *\cir<2pt>{} ="D",
	(49,0) *++!D{0} *{\bullet} ="E",	
	(56,0) *++!D{0} *+!U{\alpha_{2m}} *\cir<2pt>{} ="F",
	\ar@{-} "A";"B"
	\ar@{-} "B";"C"	
	\ar@{-} "C"; (19,0)
	\ar@{.} (19,0) ; (38,0)^*!U{\cdots}
	\ar@{-} (38,0) ; "D"
	\ar@{-} "D";"E"
	\ar@{<=} "E";"F"
\end{xy} & No \\ \hline
	$\mathfrak{so}(2k-p,p)$ & \begin{xy}
	*++!D{0} *\cir<2pt>{}        ="A",
	(7,0) *++!D{1} *\cir<2pt>{} ="B",
	(14,0) *++!D{0} *\cir<2pt>{} ="C",
	(28,0) *++!D{0} *+!U{\alpha_p} *\cir<2pt>{} ="D",
	(35,0) *++!D{0} *{\bullet} ="E",
	(49,0) *++!D{0} *{\bullet} ="F",	
	(56,-10) *++!D{0} *++!L{\alpha_{k-1}} *{\bullet} ="G",
	(56,10) *++!D{0} *++!L{\alpha_{k}} *{\bullet} ="H",
	\ar@{-} "A";"B"
	\ar@{-} "B";"C"
	\ar@{-} "C"; (19,0)
	\ar@{.} (19,0) ; (23,0)^*!U{\cdots}
	\ar@{-} (23,0) ; "D"
	\ar@{-} "D";"E"
	\ar@{-} "E"; (40,0)
	\ar@{.} (40,0) ; (44,0)^*!U{\cdots}
	\ar@{-} (44,0) ; "F"
	\ar@{-} "F";"G"
	\ar@{-} "F";"H"	
\end{xy} & No if $p=1$ \\ \hline
	$\mathfrak{so}(k+1,k-1)$ & \begin{xy}
	*++!D{0} *\cir<2pt>{}      ="A",
	(7,0) *++!D{1} *\cir<2pt>{} ="B",
	(14,0) *++!D{0} *\cir<2pt>{} ="C",
	(35,0) *++!D{0} *\cir<2pt>{} ="D",
	(42,0) *++!D{0} *\cir<2pt>{} ="E",
	(49,0) *++!D{0} *\cir<2pt>{} ="F",	
	(56,-10) *++!D{0} *++++!L{\alpha_{k-1}} *\cir<2pt>{} ="G",
	(56,10) *++!D{0} *++++!L{\alpha_k} *\cir<2pt>{} ="H",	
	\ar@{-} "A";"B"
	\ar@{-} "B";"C"	
	\ar@{-} "C"; (19,0)
	\ar@{.} (19,0) ; (30,0)^*!U{\cdots}
	\ar@{-} (30,0) ; "D"
	\ar@{-} "D";"E"
	\ar@{-} "E";"F"
	\ar@{-} "F";"G"
	\ar@{-} "F";"H"
	\ar@/_4mm/@{<->}@<-2mm> "G";"H"
\end{xy} & Yes \\ \hline
	$\mathfrak{so}(k,k)$ & \begin{xy}
	*++!D{0} *\cir<2pt>{}      ="A",
	(7,0) *++!D{1} *\cir<2pt>{} ="B",
	(14,0) *++!D{0} *\cir<2pt>{} ="C",
	(35,0) *++!D{0} *\cir<2pt>{} ="D",
	(42,0) *++!D{0} *\cir<2pt>{} ="E",
	(49,0) *++!D{0} *\cir<2pt>{} ="F",	
	(56,-10) *++!D{0} *++!L{\alpha_{k-1}} *\cir<2pt>{} ="G",
	(56,10) *++!D{0} *++!L{\alpha_k} *\cir<2pt>{} ="H",	
	\ar@{-} "A";"B"
	\ar@{-} "B";"C"	
	\ar@{-} "C"; (19,0)
	\ar@{.} (19,0) ; (30,0)^*!U{\cdots}
	\ar@{-} (30,0) ; "D"
	\ar@{-} "D";"E"
	\ar@{-} "E";"F"
	\ar@{-} "F";"G"
	\ar@{-} "F";"H"
\end{xy} & Yes \\ \hline
	$\mathfrak{so}^*(4m)$ & \begin{xy}
	*++!D{0} *{\bullet}       ="A",
	(7,0) *++!D{1} *\cir<2pt>{} ="B",
	(14,0) *++!D{0} *{\bullet} ="C",
	(35,0) *++!D{0} *\cir<2pt>{} ="D",
	(42,0) *++!D{0} *{\bullet} ="E",
	(49,0) *++!D{0} *+!L{\alpha_{2m-2}} *\cir<2pt>{} ="F",	
	(56,-10) *++!D{0} *++!U{} *{\bullet} ="G",
	(56,10) *++!D{0} *++!D{} *\cir<2pt>{} ="H",	
	\ar@{-} "A";"B"
	\ar@{-} "B";"C"	
	\ar@{-} "C"; (19,0)
	\ar@{.} (19,0) ; (30,0)^*!U{\cdots}
	\ar@{-} (30,0) ; "D"
	\ar@{-} "D";"E"
	\ar@{-} "E";"F"
	\ar@{-} "F";"G"
	\ar@{-} "F";"H"
\end{xy} & Yes \\ \hline
	$\mathfrak{so}^*(4m+2)$ & \begin{xy}
	*++!D{0} *{\bullet}       ="A",
	(7,0) *++!D{1} *\cir<2pt>{} ="B",
	(14,0) *++!D{0} *{\bullet} ="C",
	(35,0) *++!D{0} *{\bullet} ="D",
	(42,0) *++!D{0} *\cir<2pt>{} ="E",
	(49,0) *++!D{0} *+!L{} *{\bullet} ="F",	
	(56,-10) *++!D{0} *++++!L{\alpha_{2m}} *\cir<2pt>{} ="G",
	(56,10) *++!D{0} *++++!L{\alpha_{2m+1}} *\cir<2pt>{} ="H",	
	\ar@{-} "A";"B"
	\ar@{-} "B";"C"	
	\ar@{-} "C"; (19,0)
	\ar@{.} (19,0) ; (30,0)^*!U{\cdots}
	\ar@{-} (30,0) ; "D"
	\ar@{-} "D";"E"
	\ar@{-} "E";"F"
	\ar@{-} "F";"G"
	\ar@{-} "F";"H"
	\ar@/_4mm/@{<->}@<-2mm> "G";"H"	
\end{xy} & Yes \\ \hline	
	$\mathfrak{e}_{6(6)}$ & \begin{xy}
	*++!D{0} *\cir<2pt>{}        ="A",
	(10,0) *++!D{0} *\cir<2pt>{} ="B",
	(20,0) *++!D{0} *\cir<2pt>{} ="C",
	(30,0) *++!D{0} *\cir<2pt>{} ="D",
	(40,0) *++!D{0} *\cir<2pt>{} ="E",
	(20,-10) *++!L{1} *++!U{} *\cir<2pt>{} ="F",
	\ar@{-} "A";"B"
	\ar@{-} "B";"C"
	\ar@{-} "C";"D"
	\ar@{-} "D";"E"
	\ar@{-} "C";"F"
\end{xy} & Yes \\ \hline
	$\mathfrak{e}_{6(2)}$ & \begin{xy}
	*++!D{0} *\cir<2pt>{}        ="A",
	(10,0) *++!D{0} *\cir<2pt>{} ="B",
	(20,0) *+!D{0} *\cir<2pt>{} ="C",
	(30,0) *++!D{0} *\cir<2pt>{} ="D",
	(40,0) *++!D{0} *\cir<2pt>{} ="E",
	(20,-10) *++!L{1} *++!U{} *\cir<2pt>{} ="F",
	\ar@{-} "A";"B"
	\ar@{-} "B";"C"
	\ar@{-} "C";"D"
	\ar@{-} "D";"E"
	\ar@{-} "C";"F"
	\ar@(ru,lu) @<1mm> @{<->} "A" ; "E"
	\ar@/^4mm/ @<1mm> @{<->} "B" ; "D"	
\end{xy} & Yes \\ \hline
	$\mathfrak{e}_{6(-14)}$ & \begin{xy}
	*++!D{0} *\cir<2pt>{}        ="A",
	(10,0) *++!D{0} *{\bullet} ="B",
	(20,0) *++!D{0} *{\bullet} ="C",
	(30,0) *++!D{0} *{\bullet} ="D",
	(40,0) *++!D{0} *\cir<2pt>{} ="E",
	(20,-10) *++!L{1} *++!U{} *\cir<2pt>{} ="F",
	\ar@{-} "A";"B"
	\ar@{-} "B";"C"
	\ar@{-} "C";"D"
	\ar@{-} "D";"E"
	\ar@{-} "C";"F"
	\ar@(ru,lu) @<1mm> @{<->} "A" ; "E"
\end{xy} & Yes \\ \hline
	$\mathfrak{e}_{6(-26)}$ & \begin{xy}
	*++!D{0} *\cir<2pt>{}        ="A",
	(10,0) *++!D{0} *{\bullet} ="B",
	(20,0) *++!D{0} *{\bullet} ="C",
	(30,0) *++!D{0} *{\bullet} ="D",
	(40,0) *++!D{0} *\cir<2pt>{} ="E",
	(20,-10) *++!L{1} *++!U{} *{\bullet} ="F",
	\ar@{-} "A";"B"
	\ar@{-} "B";"C"
	\ar@{-} "C";"D"
	\ar@{-} "D";"E"
	\ar@{-} "C";"F"
\end{xy} & No \\ \hline
	$\mathfrak{e}_{7(7)}$ & \begin{xy}
	*++!D{0} *\cir<2pt>{}        ="A",
	(10,0) *++!D{0} *\cir<2pt>{} ="B",
	(20,0) *++!D{0} *\cir<2pt>{} ="C",
	(30,0) *++!D{0} *\cir<2pt>{} ="D",
	(40,0) *++!D{0} *\cir<2pt>{} ="E",
	(50,0) *++!D{1} *\cir<2pt>{} ="F",
	(30,-10) *++!L{0} *++!U{} *\cir<2pt>{} ="G",
	\ar@{-} "A";"B"
	\ar@{-} "B";"C"
	\ar@{-} "C";"D"
	\ar@{-} "D";"E"
	\ar@{-} "E";"F"
	\ar@{-} "D";"G"
\end{xy} & Yes \\ \hline	
	$\mathfrak{e}_{7(-5)}$ & \begin{xy}
	*++!D{0} *{\bullet}        ="A",
	(10,0) *++!D{0} *\cir<2pt>{} ="B",
	(20,0) *++!D{0} *{\bullet} ="C",
	(30,0) *++!D{0} *\cir<2pt>{} ="D",
	(40,0) *++!D{0} *\cir<2pt>{} ="E",
	(50,0) *++!D{1} *\cir<2pt>{} ="F",
	(30,-10) *++!L{0} *++!U{} *{\bullet} ="G",
	\ar@{-} "A";"B"
	\ar@{-} "B";"C"
	\ar@{-} "C";"D"
	\ar@{-} "D";"E"
	\ar@{-} "E";"F"
	\ar@{-} "D";"G"
\end{xy} & Yes \\ \hline
	$\mathfrak{e}_{7(-25)}$ & \begin{xy}
	*++!D{0} *\cir<2pt>{}        ="A",
	(10,0) *++!D{0} *\cir<2pt>{} ="B",
	(20,0) *++!D{0} *{\bullet} ="C",
	(30,0) *++!D{0} *{\bullet} ="D",
	(40,0) *++!D{0} *{\bullet} ="E",
	(50,0) *++!D{1} *\cir<2pt>{} ="F",
	(30,-10) *++!L{0} *++!U{} *{\bullet} ="G",
	\ar@{-} "A";"B"
	\ar@{-} "B";"C"
	\ar@{-} "C";"D"
	\ar@{-} "D";"E"
	\ar@{-} "E";"F"
	\ar@{-} "D";"G"
\end{xy} & Yes \\ \hline	
	$\mathfrak{e}_{8(8)}$ & \begin{xy}
	*++!D{1} *\cir<2pt>{}        ="A",
	(10,0) *++!D{0} *\cir<2pt>{} ="B",
	(20,0) *++!D{0} *\cir<2pt>{} ="C",
	(30,0) *++!D{0} *\cir<2pt>{} ="D",
	(40,0) *++!D{0} *\cir<2pt>{} ="E",
	(50,0) *++!D{0} *\cir<2pt>{} ="F",
	(60,0) *++!D{0} *\cir<2pt>{} ="G",
	(40,-10) *++!L{0} *++!U{} *\cir<2pt>{} ="H",
	\ar@{-} "A";"B"
	\ar@{-} "B";"C"
	\ar@{-} "C";"D"
	\ar@{-} "D";"E"
	\ar@{-} "E";"F"
	\ar@{-} "F";"G"
	\ar@{-} "E";"H"
\end{xy} & Yes \\ \hline
	$\mathfrak{e}_{8(-24)}$ & \begin{xy}
	*++!D{1} *\cir<2pt>{}        ="A",
	(10,0) *++!D{0} *\cir<2pt>{} ="B",
	(20,0) *++!D{0} *\cir<2pt>{} ="C",
	(30,0) *++!D{0} *{\bullet} ="D",
	(40,0) *++!D{0} *{\bullet} ="E",
	(50,0) *++!D{0} *{\bullet} ="F",
	(60,0) *++!D{0} *\cir<2pt>{} ="G",
	(40,-10) *++!L{0} *++!U{} *{\bullet} ="H",
	\ar@{-} "A";"B"
	\ar@{-} "B";"C"
	\ar@{-} "C";"D"
	\ar@{-} "D";"E"
	\ar@{-} "E";"F"
	\ar@{-} "F";"G"
	\ar@{-} "E";"H"
\end{xy} & Yes \\ \hline
	$\mathfrak{f}_{4(4)}$ & \begin{xy}
	*++!D{1} *\cir<2pt>{}        ="A",
	(10,0) *++!D{0} *\cir<2pt>{} ="B",
	(20,0) *++!D{0} *\cir<2pt>{} ="C",
	(30,0) *++!D{0} *\cir<2pt>{} ="D",
	\ar@{-} "A";"B"
	\ar@{=>} "B";"C"
	\ar@{-} "C";"D"
\end{xy} & Yes \\ \hline
	$\mathfrak{f}_{4(-20)}$ & \begin{xy}
	*++!D{1} *{\bullet}        ="A",
	(10,0) *++!D{0} *{\bullet} ="B",
	(20,0) *++!D{0} *{\bullet} ="C",
	(30,0) *++!D{0} *\cir<2pt>{} ="D",
	\ar@{-} "A";"B"
	\ar@{=>} "B";"C"
	\ar@{-} "C";"D"
\end{xy} & No \\ \hline	
	$\mathfrak{g}_{2(2)}$ & \begin{xy}
	*++!D{1} *\cir<2pt>{}        ="A",
	(10,0) *++!D{0} *\cir<2pt>{} ="B",
	\ar@3{->} "A";"B"
\end{xy} & Yes \\ 
\hline \hline 
\label{table:Satake_diagrams}
\end{longtable}
\end{example}

\subsection{Proof of Proposition \ref{prop:min_min}}\label{subsection:proofProp}

We consider the same setting on Section \ref{subsection:Satake_diagrams} and suppose that $\mathfrak{g}_\mathbb{C}$ is simple and $\mathfrak{g}$ is not compact.
In Proposition \ref{prop:min_min}, 
the equivalence between \eqref{item:min_min:min_meets_g} and \eqref{item:min_min:WDD_match} is obtained by Fact \ref{thm:nilp_match}.
In this subsection, we completes a proof of Proposition \ref{prop:min_min}
by proving the equivalence among \eqref{item:min_min:min_meets_g}, \eqref{item:min_min:one_dim_highest}, \eqref{item:min_min:real} and \eqref{item:min_min:WDD_extended}.

\begin{proof}[Proof of the equivalence between $\eqref{item:min_min:one_dim_highest}$ and $\eqref{item:min_min:real}$ in Proposition $\ref{prop:min_min}$]
Recall that 
\[
\dim_\mathbb{R} \mathfrak{g}_{\lambda} = \sharp \{\, \alpha \in \Delta(\mathfrak{g}_\mathbb{C},\mathfrak{h}_\mathbb{C}) \mid \alpha|_\mathfrak{a} = \lambda \,\}.
\]
If $\phi$ is a real root, then for each root $\alpha \in \Delta(\mathfrak{g}_\mathbb{C},\mathfrak{h}_\mathbb{C})$ except for $\phi$, 
we have $\alpha|_\mathfrak{a} \neq \lambda \ (= \phi|_\mathfrak{a})$ since $\phi$ is the longest root of $\Delta(\mathfrak{g}_\mathbb{C},\mathfrak{h}_\mathbb{C})$. Thus $\dim_\mathbb{R} \mathfrak{g}_{\lambda} =1$ in this case.
Conversely, we assume that $\phi$ is not a real root.
Let us denote by $\tau$ the anti $\mathbb{C}$-linear involution corresponding to $\mathfrak{g}_\mathbb{C} = \mathfrak{g} + \sqrt{-1}\mathfrak{g}$.
That is, $\tau$ is the complex conjugation of $\mathfrak{g}_\mathbb{C}$ 
with respect to its real form $\mathfrak{g}$.
Then $\tau$ induces the involution $\tau^*$ on $\mathfrak{h}^*$, 
and it preserves $\Delta(\mathfrak{g}_\mathbb{C},\mathfrak{h}_\mathbb{C})$.
Since $\phi|_{\sqrt{-1}\mathfrak{t}} \neq 0$, 
we obtain that $\tau^* \phi \neq \phi$ and 
$(\tau^* \phi)|_\mathfrak{a} = \phi|_\mathfrak{a} = \lambda$.
Hence, $\dim_\mathbb{R} \mathfrak{g}_{\lambda} \geq 2$.
\end{proof}

\begin{proof}[Proof of the equivalence between $\eqref{item:min_min:min_meets_g}$ and $\eqref{item:min_min:real}$ in Proposition $\ref{prop:min_min}$]
Recall that $H_{\phi} \in \mathfrak{h}$ 
is the hyperbolic element corresponding to $\minC$ 
(see Section \ref{subsection:Hyperbolic_orbits} for the notation).
Thus by Fact \ref{thm:nilp_match},
$\minC$ meets $\mathfrak{g}$ 
if and only if $H_{\phi}$ is in $\mathfrak{a}$.
By the definition of $H_{\phi}$, 
the highest root $\phi$ is real if and only if $H_{\phi}$ is in $\mathfrak{a}$.
This completes the proof.
\end{proof}

\begin{proof}[Proof of the equivalence between $\eqref{item:min_min:WDD_match}$ and $\eqref{item:min_min:WDD_extended}$ in Proposition $\ref{prop:min_min}$]
In the case where 
$\mathfrak{g}_\mathbb{C} \simeq \mathfrak{sl}(2,\mathbb{C})$, 
our non-compact real form $\mathfrak{g}$ must be isomorphic to $\mathfrak{sl}(2,\mathbb{R})$.
Then our claim holds since the Satake diagram of $\mathfrak{sl}(2,\mathbb{C})$ has no black node and matches any weighted Dynkin diagram.
Let us consider the cases where $\rank \mathfrak{g}_\mathbb{C} \geq 2$. 
In these case, as we observed in the last of Section \ref{subsection:WDD_of_nilp}, for a simple root $\alpha \in \Pi$, 
the weight on $\alpha$ for the weighted Dynkin diagram of $\minC$ is $1$ [resp.~$0$] if $\alpha$ has some edges [resp.~no edge] connected to the node $-\phi$ in the extended Dynkin diagram.
Then we only need to show that for a pair $\alpha, \beta \in \Pi$ joined by an arrow on the Satake diagram of $\mathfrak{g}$, the node $\alpha$ has some edges connected to $-\phi$ if and only if $\beta$ has some edges connected to $-\phi$.
By \cite[Lemma 2.10]{Helminck88}, 
there exists an involution $\sigma^*$ of $\mathfrak{h}^*$ such that $\sigma^* \Delta(\mathfrak{g}_\mathbb{C},\mathfrak{h}_\mathbb{C}) = \Delta(\mathfrak{g}_\mathbb{C},\mathfrak{h}_\mathbb{C})$, 
$\sigma^* \Pi = \Pi$ and $\sigma^* \alpha = \beta$.
Note that $\sigma^* \phi = \phi$
since $\phi$ is the unique longest dominant root in $\Delta^+(\mathfrak{g}_\mathbb{C},\mathfrak{h}_\mathbb{C})$.
Therefore, we have 
\begin{align*}
\langle \alpha, -\phi \rangle 
= \langle \alpha, -\sigma^* \phi \rangle 
= \langle \sigma^* \alpha, -\phi \rangle 
= \langle \beta, -\phi \rangle.
\end{align*}
This completes the proof.
\end{proof}

\section{Weighted Dynkin diagrams of $\ming$}\label{section:WDD_of_min}

Let $\mathfrak{g}_\mathbb{C}$ be a complex simple Lie algebra 
and $\mathfrak{g}$ a non-compact real form of $\mathfrak{g}_\mathbb{C}$.
In this section, we determine $\ming$ for each $\mathfrak{g}$ 
by describing the weighted Dynkin diagram of $\ming$.
Recall that Proposition \ref{prop:min_min} claims that $\minC = \ming$ if and only if $\dim_\mathbb{R} \mathfrak{g}_\lambda =1$.
Thus our concern is in the cases where $\dim_\mathbb{R} \mathfrak{g}_\lambda \geq 2$ 
i.e.~$\mathfrak{g}$ is isomorphic to one of $\mathfrak{su}^*(2k)$, $\mathfrak{so}(n,1)$, $\mathfrak{sp}(p,q)$, $\mathfrak{e}_{6(-26)}$ or $\mathfrak{f}_{4(-20)}$.

We use the same notation in Section \ref{subsection:Satake_diagrams} 
and assume that $\mathfrak{g}_\mathbb{C}$ is simple and $\mathfrak{g}$ is non-compact.
Let us denote by \[
\mathfrak{a}_+ := \{\, A \in \mathfrak{a} \mid \xi(A) \geq 0 \ 
\text{for any } \xi \in \Sigma^+(\mathfrak{g},\mathfrak{a}) \,\}.
\]
Then $\mathfrak{a}_+$ is a fundamental domain of $\mathfrak{a}$ for the action of $W(\mathfrak{g},\mathfrak{a})$.
Since 
\[
\Sigma^+(\mathfrak{g},\mathfrak{a}) = \{\, \alpha|_\mathfrak{a} \mid \alpha \in \Delta(\mathfrak{g}_\mathbb{C},\mathfrak{h}_\mathbb{C}) \,\} \setminus \{0\},
\]
we have $\mathfrak{a}_+ = \mathfrak{h}_+ \cap \mathfrak{a}$.

Recall that $\lambda$ is dominant by Lemma \ref{lem:restricted_highest_root}.
Thus the coroot $A_{\lambda}$ is in $\mathfrak{a}_+ (\subset \mathfrak{h}_+)$.
Therefore, $A_{\lambda}$ is the hyperbolic element corresponding to $\ming$ since we can find $X_\lambda \in \mathfrak{g}_\lambda$, $Y_\lambda \in \mathfrak{g}_{-\lambda}$ such that the triple $(A_{\lambda},X_\lambda,Y_\lambda)$ is an $\mathfrak{sl}_2$-triple in $\mathfrak{g}_\mathbb{C}$.
Therefore, to determine the weighted Dynkin diagram of $\ming$, we need to compute the weighted Dynkin diagram corresponding to $A_{\lambda}$.

Our first purpose of this section 
is to show the following proposition 
which gives a formula of $A_{\lambda}$ by $H_{\phi}$,
where $H_\phi$ is the hyperbolic element corresponding to 
$\minC$ (see Section \ref{subsection:Hyperbolic_orbits}).

\begin{prop}\label{prop:A_lambda_and_H_phi}
We denote by $\tau$ the anti $\mathbb{C}$-linear involution 
corresponding to $\mathfrak{g}_\mathbb{C} = \mathfrak{g} + \sqrt{-1}\mathfrak{g}$, 
i.e.~$\tau$ is the complex conjugation of $\mathfrak{g}_\mathbb{C}$ 
with respect to the real form $\mathfrak{g}$.
Then 
\begin{align*}
A_{\lambda} = 
\begin{cases}
H_{\phi} \quad \text{if $\dim_\mathbb{R} \mathfrak{g}_\lambda = 1$}, \\
H_{\phi} + \tau H_{\phi} \quad \text{if $\dim_\mathbb{R} \mathfrak{g}_\lambda \geq 2$}.
\end{cases}
\end{align*}
\end{prop}

In particular, 
if $\dim_\mathbb{R} \mathfrak{g}_\lambda \geq 2$, 
then the weighted Dynkin diagram of $\ming$ 
can be computed by the sum of the weighted Dynkin diagrams corresponding to 
$H_{\phi}$, i.e.~the weighted Dynkin diagram of $\minC$, 
and that corresponding to $\tau H_{\phi}$.

We compute the weighted Dynkin diagram corresponding to $A_{\lambda}$ 
for each $\mathfrak{g}$ with $\dim_\mathbb{R} \mathfrak{g}_\lambda \geq 2$ 
in Section \ref{subsection:WDD_of_O_min_g}.

\subsection{Proof of Proposition \ref{prop:A_lambda_and_H_phi}}

Recall that Proposition \ref{prop:min_min} claims that 
$\dim_\mathbb{R} \mathfrak{g}_\lambda = 1$ if and only if 
the highest root $\phi$ of $\Delta(\mathfrak{g}_\mathbb{C},\mathfrak{h}_\mathbb{C})$ is real,
i.e.~$\phi|_{\sqrt{-1}\mathfrak{t}} = 0$.
We give a proof of Proposition \ref{prop:A_lambda_and_H_phi} 
as a sequence of the following two lemmas:

\begin{lem}\label{lem:culc_of_A_lambda}
\[
A_{\lambda} = \frac{\langle \phi, \phi \rangle}{2 \langle \lambda, \lambda \rangle} 
(H_{\phi} + \tau H_{\phi}),
\]
where $\langle~,~\rangle$ is the inner product on $\mathfrak{h}^*$ and on $\mathfrak{a}^*$
induced by the Killing form $B_\C$ on $\mathfrak{g}_\C$.
In particular, if $\phi$ is a real root, then $A_{\lambda} = H_{\phi}$.
\end{lem}

\begin{lem}\label{lem:norm_of_phi}
Suppose that $\phi$ is not a real root.
Then $\langle \phi, \phi \rangle = 2 \langle \lambda,\lambda \rangle$.
\end{lem}

\begin{proof}[Proof of Lemma $\ref{lem:culc_of_A_lambda}$]
We consider $\mathfrak{h}^*$ as $\mathfrak{a}^* + \sqrt{-1}\mathfrak{t}^*$.
Then for each $\xi \in \mathfrak{a}^*$, 
\begin{align*}
\xi (\frac{\langle \phi, \phi\rangle}{2 \langle \lambda, \lambda \rangle} (H_{\phi} + \tau H_{\phi})) 
	&= \frac{\langle \phi, \phi\rangle}{\langle \lambda, \lambda \rangle} \xi(H_{\phi}) \quad (\text{since } \xi(H_{\phi}) = \xi(\tau H_{\phi})) \\
	&= \frac{2\langle \xi,\phi \rangle}{\langle \lambda, \lambda \rangle} \quad (\text{by the definition of $H_{\phi}$})\\
	&= \frac{2\langle \xi,\lambda \rangle}{\langle \lambda, \lambda \rangle} \quad (\text{since } \phi|_\mathfrak{a} = \lambda) \\
	&= \xi(A_\lambda).
\end{align*}
This completes the proof.
\end{proof}

\begin{proof}[Proof of Lemma $\ref{lem:norm_of_phi}$]
We write $\tau^*$ for the involution on $\mathfrak{h}^*$ induced by $\tau$.
It is enough to show that 
$
\langle \phi, \tau^* \phi \rangle = 0
$
because $\lambda = \frac{1}{2}(\phi + \tau^* \phi)$.
By \cite[Proposition 1.3]{Araki62}, 
$\tau^*$ is a normal involution of $\Delta(\mathfrak{g}_\mathbb{C},\mathfrak{h}_\mathbb{C})$,
i.e.~for each root $\alpha \in \Delta(\mathfrak{g}_\mathbb{C},\mathfrak{h}_\mathbb{C})$, 
the element $\alpha-\tau^* \alpha$ is not a root of $\Delta(\mathfrak{g}_\mathbb{C},\mathfrak{h}_\mathbb{C})$.
In particular, 
for any root $\alpha \in \Delta(\mathfrak{g}_\mathbb{C},\mathfrak{h}_\mathbb{C})$ 
with $\tau^* \alpha \neq \alpha$, we have 
$
\langle \alpha,\tau^* \alpha \rangle \leq 0.
$
Recall that we are assuming that $\phi$ is not real. 
Thus $\phi \neq \tau^* \phi$, and then 
$
\langle \phi,\tau^* \phi \rangle \leq 0.
$
The root $\tau^* \phi$ is in $\Delta^+(\mathfrak{g}_\mathbb{C},\mathfrak{h}_\mathbb{C})$ 
since the ordering on $\mathfrak{h}$ is an extension of the ordering on $\mathfrak{a}$.
Then we also obtain that 
$
\langle \phi,\tau^* \phi \rangle \geq 0
$
since the highest root $\phi$ is dominant.
Therefore, 
$
\langle \phi, \tau^* \phi \rangle = 0.
$
\end{proof}

\subsection{Weighted Dynkin diagrams of $\ming$}\label{subsection:WDD_of_O_min_g}

We now determine the weighted Dynkin diagram of $\ming$ 
for each $\mathfrak{g}$ with $\dim_\mathbb{R} \mathfrak{g}_\lambda \geq 2$, 
i.e.~$\mathfrak{g}$ is isomorphic to one of $\mathfrak{su}^*(2k)$, $\mathfrak{so}(n,1)$, $\mathfrak{sp}(p,q)$, $\mathfrak{e}_{6(-26)}$ or $\mathfrak{f}_{4(-20)}$.
By Proposition \ref{prop:A_lambda_and_H_phi}, 
our goal is to compute 
the weighted Dynkin diagram corresponding to $A_{\lambda} = H_{\phi} + \tau H_{\phi}$.

For simplicity, we denote by $S$ the Satake diagram of $\mathfrak{g}$.
For each simple root $\alpha$ in $\Pi$, 
we denote by $H_{\alpha} \in \mathfrak{h}$ the coroot of $\alpha$.

Then the next lemma holds:

\begin{lem}\label{lem:basis_of_it}
The set 
\[
\{\, H_{\alpha} \mid \text{$\alpha$ is black in $S$} \,\} \sqcup
\{\, H_{\alpha} - H_{\beta} \mid \text{$\alpha$ and $\beta$ are joined by an arrow in $S$} \,\}
\]
becomes a basis of $\sqrt{-1}\mathfrak{t}$.
\end{lem}

\begin{proof}
We denote by 
\begin{multline*}
\Omega = \{\, H_{\alpha} \mid \text{$\alpha$ is black in $S$} \,\}  \\ \sqcup \{\, H_{\alpha} - H_{\beta} \mid \text{$\alpha$ and $\beta$ are joined by an arrow in $S$} \,\}.
\end{multline*}
It is known that there is no triple $\{ \alpha,\beta,\gamma \}$ in $\Pi \setminus \Pi_0$ such that $\alpha|_\mathfrak{a} = \beta|_\mathfrak{a} = \gamma|_\mathfrak{a}$ 
(this fact can be found in \cite[Section 2.8]{Araki62}).
Thus $\Omega$ is linearly independent and \[
\sharp \Omega = \sharp \Pi - \sharp \overline{\Pi}.
\]
Recall that $\overline{\Pi}$ is a simple system of the restricted root system $\Sigma(\mathfrak{g},\mathfrak{a})$, 
we have $\dim_\mathbb{R} \mathfrak{a} = \sharp \overline{\Pi}$.
Since $\dim_\mathbb{R} \mathfrak{h} = \sharp \Pi$ and $\sqrt{-1}\mathfrak{t}$ is the orthogonal complement space of $\mathfrak{a}$ in $\mathfrak{h}$ for the Killing form $B_\mathbb{C}$ on $\mathfrak{g}_\mathbb{C}$, 
it remains to prove that 
\[
B_\mathbb{C}(H',A) = 0 \quad \text{for any } H' \in \Omega,\ A \in \mathfrak{a}.
\]
Let us take $\alpha \in \Pi_0$, i.e.~$\alpha$ is a black node in $S$. 
Since $\alpha|_\mathfrak{a} = 0$, we have 
\[
B_\mathbb{C}(H_{\alpha},A) = \frac{2 \alpha(A)}{\langle \alpha,\alpha \rangle} =0 \quad 
\text{for any } A \in \mathfrak{a}.
\]
Furthermore, by \cite[Lemma 2.10]{Helminck88}, 
there exists an involution $\sigma^*$ of $\mathfrak{h}^*$ such that 
$\sigma^* \alpha = \beta$ for all pair $\alpha, \beta \in \Pi \setminus \Pi_0$ 
such that $\alpha|_\mathfrak{a}= \beta|_\mathfrak{a}$, 
i.e.~$\alpha$ and $\beta$ is joined by an arrow in $S$.
In particular $|\alpha| = |\beta|$ for such pair.
Thus for any $A \in \mathfrak{a}$, we have 
\begin{align*}
B_\mathbb{C} (H_{\alpha} - H_{\beta}, A) 
	&= \frac{2\alpha(A)}{\langle \alpha, \alpha \rangle} - \frac{2\beta(A)}{\langle \beta, \beta \rangle} \\
	&= 0 \quad (\text{since } \alpha|_\mathfrak{a} = \beta|_\mathfrak{a} \text{ and } |\alpha| = |\beta|).
\end{align*}
This completes the proof.
\end{proof}

By using Lemma \ref{lem:basis_of_it}, 
we shall compute the weighted Dynkin diagram corresponding to $A_{\lambda} = H_{\phi} + \tau H_{\phi}$.
In this paper, we only give the computation 
for the case $\mathfrak{g} = \mathfrak{e}_{6(-26)}$ below.
For the other $\mathfrak{g}$ with $\dim_\mathbb{R} \mathfrak{g}_\lambda \geq 2$, 
we can compute the weighted Dynkin diagram corresponding to $A_{\lambda}$ by the same way.

\begin{example}
Let $(\mathfrak{g}_\mathbb{C},\mathfrak{g}) = (\mathfrak{e}_{6,\mathbb{C}}, \mathfrak{e}_{6(-26)})$.
We denote the Satake diagram of $\mathfrak{e}_{6(-26)}$ by 
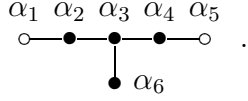
\[
\begin{xy}
	*++!D{\alpha_1} *\cir<2pt>{}        ="A",
	(6,0) *++!D{\alpha_2} *{\bullet} ="B",
	(12,0) *++!D{\alpha_3}  *{\bullet} ="C",
	(18,0) *++!D{\alpha_4} *{\bullet} ="D",
	(24,0) *++!D{\alpha_5} *\cir<2pt>{} ="E",
	(12,-6) *++!L{\alpha_6} *{\bullet} ="F",
	\ar@{-} "A";"B"
	\ar@{-} "B";"C"
	\ar@{-} "C";"D"
	\ar@{-} "D";"E"
	\ar@{-} "C";"F"
\end{xy}.
\]
By Table $\ref{table:Satake_diagrams}$, the weighted Dynkin diagram corresponding to $H_{\phi}$ is 
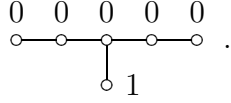
\[
\begin{xy}
	*++!D{0} *\cir<2pt>{}        ="A",
	(6,0) *++!D{0} *\cir<2pt>{} ="B",
	(12,0) *++!D{0}  *\cir<2pt>{} ="C",
	(18,0) *++!D{0} *\cir<2pt>{} ="D",
	(24,0) *++!D{0} *\cir<2pt>{} ="E",
	(12,-6) *++!L{1} *\cir<2pt>{} ="F",
	\ar@{-} "A";"B"
	\ar@{-} "B";"C"
	\ar@{-} "C";"D"
	\ar@{-} "D";"E"
	\ar@{-} "C";"F"
\end{xy}.
\]
We now compute the weighted Dynkin diagram corresponding to $A_{\lambda} = H_{\phi} + \tau H_{\phi}$.
By Fact $\ref{fact:a_to_Satake}$, 
the weighted Dynkin diagram corresponding to $A_{\lambda}$ matches the Satake diagram of $\mathfrak{e}_{6(-26)}$.
Thus we can put the weighted Dynkin diagram corresponding to $A_{\lambda}$ as
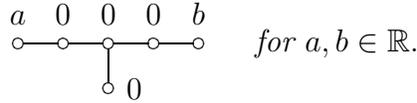
\[
\begin{xy}
	*++!D{a} *\cir<2pt>{}        ="A",
	(6,0) *++!D{0} *\cir<2pt>{} ="B",
	(12,0) *++!D{0}  *\cir<2pt>{} ="C",
	(18,0) *++!D{0} *\cir<2pt>{} ="D",
	(24,0) *++!D{b} *\cir<2pt>{} ="E",
	(12,-6) *++!L{0} *\cir<2pt>{} ="F",
	\ar@{-} "A";"B"
	\ar@{-} "B";"C"
	\ar@{-} "C";"D"
	\ar@{-} "D";"E"
	\ar@{-} "C";"F"
\end{xy} \quad \text{for } a,b \in \mathbb{R}.
\]
To determine $a,b \in \mathbb{R}$, we also put \[
H_{\phi}^{\imag} := H_{\phi} - \tau H_{\phi} \in \sqrt{-1}\mathfrak{t}.
\]
Since $A_{\lambda} + H_{\phi}^{\imag} = 2H_{\phi}$, 
the weighted Dynkin diagram corresponding to $H_{\phi}^{\imag}$ can be written by
\[
\begin{xy}
	*++!D{-a} *\cir<2pt>{}        ="A",
	(6,0) *++!D{0} *\cir<2pt>{} ="B",
	(12,0) *++!D{0}  *\cir<2pt>{} ="C",
	(18,0) *++!D{0} *\cir<2pt>{} ="D",
	(24,0) *++!D{-b} *\cir<2pt>{} ="E",
	(12,-6) *++!L{2} *\cir<2pt>{} ="F",
	\ar@{-} "A";"B"
	\ar@{-} "B";"C"
	\ar@{-} "C";"D"
	\ar@{-} "D";"E"
	\ar@{-} "C";"F"
\end{xy}.
\]
That is, 
we have 
\begin{align*}
&\alpha_1(H_{\phi}^{\imag}) = -a,\\
&\alpha_2(H_{\phi}^{\imag}) = \alpha_3(H_{\phi}^{\imag}) = \alpha_4(H_{\phi}^{\imag}) =0, \\
&\alpha_5(H_{\phi}^{\imag}) = -b, \\
&\alpha_6(H_{\phi}^{\imag}) = 2.
\end{align*}
By Lemma $\ref{lem:basis_of_it}$, the set $\{\, H_{\alpha_2},H_{\alpha_3},H_{\alpha_4},H_{\alpha_6} \,\}$ becomes a basis of $\sqrt{-1}\mathfrak{t}$.
Thus $H_{\phi}^\imag \in \sqrt{-1}\mathfrak{t}$ can be written by 
\[
H_{\phi}^\imag = c_2 H_{\alpha_2} + c_3 H_{\alpha_3} + c_4 H_{\alpha_4} + c_6 H_{\alpha_6} \quad \text{for } c_2,c_3,c_4,c_6 \in \mathbb{R}.
\]
Since 
$
\alpha_i(H_{\alpha_j}) = 2 \langle \alpha_i,\alpha_j \rangle/\langle \alpha_j,\alpha_j \rangle,
$
by comparing with the Dynkin diagram of $\mathfrak{e}_{6,\mathbb{C}}$, we obtain
\begin{align*}
\alpha_1(H_{\phi}^{\imag}) & = -c_2, \\
\alpha_2(H_{\phi}^{\imag}) & = 2c_2 -c_3,\\
\alpha_3(H_{\phi}^{\imag}) & = -c_2 + 2c_3 - c_4 - c_6,\\
\alpha_4(H_{\phi}^{\imag}) & = -c_3 + 2c_4,\\
\alpha_5(H_{\phi}^{\imag}) & = -c_4, \\
\alpha_6(H_{\phi}^{\imag}) & = -c_3 + 2c_6.
\end{align*}
Hence $a=b=1$.
Therefore, the weighted Dynkin diagram of $\ming$ for 
$\mathfrak{g} = \mathfrak{e}_{6(-26)}$ is 
\[
\begin{xy}
	*++!D{1} *\cir<2pt>{}        ="A",
	(6,0) *++!D{0} *\cir<2pt>{} ="B",
	(12,0) *++!D{0}  *\cir<2pt>{} ="C",
	(18,0) *++!D{0} *\cir<2pt>{} ="D",
	(24,0) *++!D{1} *\cir<2pt>{} ="E",
	(12,-6) *++!L{0} *\cir<2pt>{} ="F",
	\ar@{-} "A";"B"
	\ar@{-} "B";"C"
	\ar@{-} "C";"D"
	\ar@{-} "D";"E"
	\ar@{-} "C";"F"
\end{xy}.
\]
\end{example}

The result of our computations for all $\mathfrak{g}$ with $\dim_\mathbb{R} \mathfrak{g}_\lambda \geq 2$ 
is in Table \ref{table:O_min_g} in Section \ref{section:intro}.

\begin{rem}\label{rem:E6_Djokovic}
The weighted Dynkin diagram of $\ming$ for each $\mathfrak{g}$ with $\dim_\R \mathfrak{g}_\lambda \geq 2$ can be determined by the classification result of real nilpotent orbits in $\mathfrak{g}$.
For example, let us consider the case where $\mathfrak{g} = \mathfrak{e}_{6(-26)}$ as follows.
Djokovic \cite{Djokovic88E6} proved that there exist only two non-trivial real nilpotent orbits in $\mathfrak{e}_{6(-26)}$. 
The list of real nilpotent orbits in $\mathfrak{e}_{6(-26)}$ can be found in the table in \cite[Chapter 9.6]{Collingwood-McGovern93} and 
the weighted Dynkin diagram of the complexification of each orbit is described in the first column of the table.
Recall that the real dimension of a real nilpotent orbit and the complex dimension of its complexification are the same.
In a table in \cite[Chapter 8.4]{Collingwood-McGovern93},
the complex dimensions of the complex nilpotent orbits in $\mathfrak{e}_{6,\mathbb{C}}$ corresponding to the label $1$ and $2$ can be found as $32$ and $48$, respectively.
Therefore, $\mathfrak{e}_{6(-26)}$ has a two real nilpotent orbits $\Orbit_1$ and $\Orbit_2$ with $\dim_\mathbb{R} \Orbit_1 = 32$ and $\dim_\mathbb{R} \Orbit_2 = 48$, respectively.
In particular, $\mathfrak{e}_{6(-26)}$ has the unique real nilpotent orbit $\Orbit_1$ with the minimal positive dimension, 
and the weighted Dynkin diagram of its complexification is 
\[
\begin{xy}
	*++!D{1} *\cir<2pt>{}        ="A",
	(6,0) *++!D{0} *\cir<2pt>{} ="B",
	(12,0) *++!D{0}  *\cir<2pt>{} ="C",
	(18,0) *++!D{0} *\cir<2pt>{} ="D",
	(24,0) *++!D{1} *\cir<2pt>{} ="E",
	(12,-6) *++!L{0} *\cir<2pt>{} ="F",
	\ar@{-} "A";"B"
	\ar@{-} "B";"C"
	\ar@{-} "C";"D"
	\ar@{-} "D";"E"
	\ar@{-} "C";"F"
\end{xy}.
\]
Furthermore, 
by the Hasse diagram of complex nilpotent orbits in $\mathfrak{e}_{6,\mathbb{C}}$, which can be found in \cite[\S 13.4]{Carter93finite}
we can observe that 
the complexification of $\Orbit_1$ is contained in the closure of the complexification of $\Orbit_2$.
Thus the complexification of $\Orbit_1$ is minimal in $\mathcal{N}_{\mathfrak{e}_{6(-26)}}/{G_\mathbb{C}}$ except for the zero-orbit,
and hence, the complexification of $\Orbit_1$ is our $\ming$ in this case.
\end{rem}

\section{$G$-orbits in $\ming \cap \mathfrak{g}$}\label{section:real-orbits_in_min}

Let $\mathfrak{g}_\mathbb{C}$ be a complex simple Lie algebra and $\mathfrak{g}$ a non-compact real form of $\mathfrak{g}_\mathbb{C}$.
A proof of Thorem \ref{thm:number_of_G-orbits_in_minimal} is given in this section.

Throughout this section, 
we take $G$ for the connected linear Lie group with its Lie algebra $\mathfrak{g}$ 
and $G_\mathbb{C}$ the complexification of $G$.
We also fix a Cartan decomposition $\mathfrak{g} = \mathfrak{k} + \mathfrak{p}$ of $\mathfrak{g}$,
and write $K$ for the maximal compact subgroup of $G$ with its Lie algebra $\mathfrak{k}$. Note that $K$ is not connected in some cases. 
We take a maximal abelian subspace $\mathfrak{a}$ of $\mathfrak{p}$ and fix an ordering on $\mathfrak{a}$.
Let $\lambda$ be the highest root of $\Sigma(\mathfrak{g},\mathfrak{a})$ 
with respect to the ordering on $\mathfrak{a}$.
Let us denote by $M := Z_K(\mathfrak{a})$ and $A := \exp \mathfrak{a}$.
Then the closed subgroup $MA$ of $G$ coincides with $Z_G(\mathfrak{a})$.
Thus $MA$ acts on the highest root space $\mathfrak{g}_\lambda$ by the adjoint action.

Our purpose in this section is to show 
the following three propositions:

\begin{prop}\label{prop:MA-orbits_to_G-orbits}
The map
\begin{align*}
\{\, \text{non-zero $MA$-orbits in $\mathfrak{g}_\lambda$} \,\} &\rightarrow 
\{\, \text{$G$-orbits in $\ming \cap \mathfrak{g}$} \,\}
\\ 
\Orbit^{MA} &\mapsto G \cdot \Orbit^{MA}
\end{align*}
is bijective.
\end{prop}

\begin{prop}\label{prop:MA-orbits_in_g_lambda_2}
Suppose that $\dim_\mathbb{R} \mathfrak{g}_\lambda \geq 2$.
Then $\mathfrak{g}_\lambda \setminus \{0\}$ becomes a single $MA$-orbit.
\end{prop}

\begin{prop}\label{prop:MA-orbits_in_g_lambda_1}
Suppose that $\dim_\mathbb{R} \mathfrak{g}_\lambda = 1$.
Then the following holds$:$
\begin{enumerate}
\item If $(\mathfrak{g},\mathfrak{k})$ is of non-Hermitian type,
then $\mathfrak{g}_\lambda \setminus \{0\}$ becomes a single $($disconnected$)$ $MA$-orbit.
\item If $(\mathfrak{g},\mathfrak{k})$ is of Hermitian type,
then $\mathfrak{g}_\lambda \setminus \{0\}$ split into two connected $MA$-orbits.
\end{enumerate}
\end{prop}

By combining the classification \eqref{item:min_min:classification} in Proposition \ref{prop:min_min} with the list of simple Lie algebras of non-Hermitian type, 
we see that $\minC \neq \ming$ only if $(\mathfrak{g},\mathfrak{k})$ is of non-Hermitian type.
Therefore, Theorem \ref{thm:number_of_G-orbits_in_minimal} follows from above propositions immediately.

\subsection{Bijection between the set of non-zero $MA$-orbits in $\mathfrak{g}_\lambda$ and the set $G$-orbits in $\ming \cap \mathfrak{g}$}

We prove Proposition \ref{prop:MA-orbits_to_G-orbits} in this subsection.

By Theorem \ref{thm:O_min_g}, 
the orbit $\ming$ can be written by 
\[
\ming = G_\mathbb{C} \cdot (\mathfrak{g}_\lambda \setminus \{0\}),
\]
and by Proposition \ref{prop:G-orbits_meets_highest}, 
any $G$-orbit in $\ming \cap \mathfrak{g}$ meets $\mathfrak{g}_\lambda \setminus \{0\}$. 
Thus the map in Proposition \ref{prop:MA-orbits_to_G-orbits} is well-defined and surjective.
Therefore, the proof of Proposition \ref{prop:MA-orbits_to_G-orbits} is reduced to show that: 
For $X_\lambda, X_\lambda' \in \mathfrak{g}_\lambda$, if there exists $g \in G$ such that $g X_\lambda = X_\lambda'$, then there exists $m \in M$ and $a \in A$ such that $ma X_\lambda = X_\lambda'$.

We prove the claim above dividing into two lemmas below:

\begin{lem}\label{lem:G-conju_to_N_G(a)-conju}
For $X_\lambda \in \mathfrak{g}_\lambda \setminus \{0\}$ and $g \in G$,
if $g X_\lambda$ is also in $\mathfrak{g}_\lambda$, 
then there exists $m' \in N_K(\mathfrak{a})$ and $a \in A$ such that $m'a X_\lambda = g X_\lambda$.
\end{lem}

\begin{lem}\label{lem:N_G-congugate_in_g_lambda}
For $X_\lambda \in \mathfrak{g}_\lambda \setminus \{0\}$ and $m' \in N_K(\mathfrak{a})$,
if $m' X_{\lambda}$ is also in $\mathfrak{g}_{\lambda}$, 
then there exists $m \in M \ (= Z_K(\mathfrak{a}))$ such that $m X_\lambda = m' X_\lambda$.
\end{lem}

\begin{proof}[Proof of Lemma $\ref{lem:G-conju_to_N_G(a)-conju}$]
For simplicity, we put $X_{\lambda}' := g X_{\lambda}$.
Since $N_G(\mathfrak{a}) = N_K(\mathfrak{a}) A$, it is enough to find $g' \in G$ such that $g' X_\lambda = X_\lambda'$ and $g' \mathfrak{a} = \mathfrak{a}$.
Let $A_{\lambda}$ be the coroot of $\lambda$ in $\mathfrak{a}$.
Then by Lemma \ref{lem:sl_2_in_g},
there exists $Y_\lambda, Y_\lambda' \in \mathfrak{g}_{-\lambda}$ such that 
$( A_{\lambda}, X_\lambda, Y_\lambda )$ and $( A_{\lambda}, X_\lambda', Y_\lambda' )$ are both $\mathfrak{sl}_2$-triples in $\mathfrak{g}$.
Since $g$ is an automorphism of $\mathfrak{g}$ and $g X_\lambda = X_\lambda'$, the triple
$(gA_{\lambda}, X_\lambda', gY_\lambda)$
is also an $\mathfrak{sl}_2$-triple in $\mathfrak{g}$.
In particular, $( A_{\lambda}, X_\lambda', Y_\lambda' )$ and $( g A_{\lambda}, X_\lambda', g Y_\lambda )$ are both $\mathfrak{sl}_2$-triples in $\mathfrak{g}$ with the same nilpotent element.
Therefore, by Kostant's theorem for $\mathfrak{sl}_2$-triples with the same nilpotent element in a semisimple Lie algebra, there exists an element $g_1 \in G$ such that \[
g_1 (g A_{\lambda}) = A_{\lambda}, \ 
g_1 X_\lambda' = X_\lambda' 
\text{ and } 
g_1 (g Y_\lambda) = Y_\lambda'.
\] 
Write $g_2 := g_1 \cdot g$.
Then 
\[
g_2 A_{\lambda} = A_{\lambda},\
 g_2 X_\lambda = X_\lambda' 
\text{ and } 
g_2 Y_\lambda = Y_\lambda'.
\]
Recall that $\mathfrak{a} = \mathbb{R} A_{\lambda} \oplus \Ker \lambda$.
If we find $g_3 \in G$ such that 
\[
g_3  (g_2  \Ker \lambda) = \Ker \lambda,\
g_3 A_{\lambda} = A_{\lambda}
\text{ and }
g_3 X_\lambda' = X_\lambda',
\]
then we can take $g'$ as $g_3 \cdot g_2$. 
We shall find such $g_3$.
Let us denote by 
$
\mathfrak{l}' = \mathbb{R}\text{-span} \langle A_{\lambda}, X_\lambda',Y_\lambda' \rangle
$
the subalgebra spaned by the $\mathfrak{sl}_2$-triple $(A_{\lambda},X_\lambda',Y_\lambda')$.
Then there exists a Cartan involution $\theta'$ on $\mathfrak{g}$ preserving $\mathfrak{l}'$ by Mostow's theorem \cite[Theorem 6]{Mostow55}.
We set 
\begin{align*}
\mathfrak{g}_0 
	:&= Z_\mathfrak{g}(\mathfrak{l}') \\
	&= \{\, X \in \mathfrak{g} \mid [X,A_{\lambda}] = [X,X_\lambda']  =0 \,\},
\end{align*}
where the second equation can be obtained by the representation theory of $\mathfrak{sl}(2,\mathbb{C})$.
We note that $\mathfrak{g}_0$ is a reductive subalgebra of $\mathfrak{g}$ since the Cartan involution $\theta'$ preserves $\mathfrak{g}_0$.
The subspace $\Ker \lambda$ of $\mathfrak{a}$ is contained in $\mathfrak{g}_0$ 
since $[\Ker \lambda, \mathfrak{l}'] = \{0\}$.
In particular, $\Ker \lambda$ becomes a maximally split abelian subspace of $\mathfrak{g}_0$.
We have 
\begin{align*}
&[g_2 \Ker \lambda, A_{\lambda}] = g_2 [\Ker \lambda, A_{\lambda}] = \{0\}, \\
&[g_2 \Ker \lambda, X_\lambda'] = g_2 [\Ker \lambda, X_\lambda] = \{0\}.
\end{align*}
Thus the subspace $g_2 \Ker \lambda$ of $g_2 \mathfrak{a}$ is also contained in $\mathfrak{g}_0$ and becomes a maximally split abelian subspace of $\mathfrak{g}_0$.
Let us write $G_0$ for the analytic subgroup of $G$ with its Lie algebra $\mathfrak{g}_0$.
Recall that any two maximally split abelian subalgebras of $\mathfrak{g}_0$ are $G_0$-conjugate. 
Then there exists $g_3 \in G_0 \subset G$ such that \[
g_3 (g_2 \Ker \lambda) = \Ker \lambda,
\]
and hence $g_3 A_{\lambda} = A_{\lambda}$ and $g_3 X_\lambda' = X_\lambda'$.
\end{proof}

To prove Lemma \ref{lem:N_G-congugate_in_g_lambda},
we need the following lemma for Weyl groups of root systems:
\begin{lem}\label{lem:Weyl_invariant}
Let $\Sigma$ be a root system realized in a vector space $V$ with an inner product $\langle \ ,\ \rangle$, 
and $W(\Sigma)$ the Weyl group of $\Sigma$ acting on $V$.
We fix a positive system $\Sigma^+$ of $\Sigma$, 
and write $\Pi$ for the simple system of $\Sigma^+$.
Let $v$ be a dominant vector, 
i.e.~$\langle \alpha, v \rangle \geq 0$ 
for any $\alpha \in \Sigma^+$, 
and $w \in W(\Sigma)$ with $w \cdot v =v$.
Then there exists a sequence $s_1,\dots,s_{l}$
of root reflections with $s_i \cdot v = v$ for any $i = 1,\dots,l$ 
such that 
\[
w = s_1 s_{2} \cdots s_{l}.
\]
\end{lem}

\begin{proof}[Proof of Lemma \ref{lem:Weyl_invariant}]
Let $n := | \Sigma^+ \setminus w \Sigma^+|$.
We prove our claim by the induction of $n$.
If $n=0$, then $\Sigma^+ = w \Sigma^+$.
Thus $w \Pi = \Pi$ and $w = \Id_V$.
We assume that $n \geq 1$. Then $\Pi \setminus w \Sigma^+ \neq \emptyset$.
It suffice to show that 
any simple roots $\alpha \in \Pi \setminus w \Sigma^+$ satisfies that $\langle \alpha ,v \rangle =0$ and 
$|\Sigma^+ \setminus (s_\alpha w) \Sigma^+| \leq n-1$.
Since $w^{-1}\alpha \not \in \Sigma^+$, that is, $w^{-1}\alpha$ is a negative root,
we obtain that $\langle w^{-1}\alpha, v \rangle \leq 0$.
Combining $\langle \alpha, v \rangle \geq 0$ with $w \cdot v =v$, we have $\langle \alpha, v \rangle = 0$.
To complete the proof, we shall show the following: 
\begin{itemize}
\item For any $\beta \in w\Sigma^+ \cap \Sigma^+$, the root $s_\alpha \beta$ is  also in $\Sigma^+$.
\item There exists $\gamma \in w\Sigma^+ \setminus \Sigma^+$ such that $s_\alpha \gamma$ is in $\Sigma^+$.
\end{itemize}
In general, for any positive root $\beta \in \Sigma^+$ except for $\alpha$ or $2\alpha$, the root $s_\alpha \beta$ is also positive.
Thus for any $\beta \in w\Sigma^+ \cap \Sigma^+$, the root $s_\alpha \beta$ is in $\Sigma^+$ since $\alpha$ and $2\alpha$ are both not in $w \Sigma^+$. 
Thus the first one of our claims holds.
We take $\gamma := -\alpha$. 
Then $\gamma$ is in $w \Sigma^+ \setminus \Sigma^+$ since $\alpha$ is in $\Sigma^+ \setminus w\Sigma^+$.
Furthermore, $s_\alpha \gamma = \alpha$ is in $\Sigma^+$.
Thus the second one of our claims also holds.
Combining the claims above, we obtain that \[
|\Sigma^+ \cap w \Sigma^+| < |\Sigma^+ \cap (s_\alpha  w)\Sigma^+|,
\]
and hence $| \Sigma^+ \setminus (s_\alpha w) \Sigma^+| \leq n-1$.
\end{proof}

Let us give a proof of Lemma \ref{lem:N_G-congugate_in_g_lambda} as follows.

\begin{proof}[Proof of Lemma $\ref{lem:N_G-congugate_in_g_lambda}$]
We denote the element of $W(\mathfrak{g},\mathfrak{a}) = N_K(\mathfrak{a}) / Z_K(\mathfrak{a})$ corresponding to $m' \in N_K(\mathfrak{a})$ by $w$.
Then $w \lambda = \lambda$ since $m' \mathfrak{g}_\lambda = \mathfrak{g}_{w \lambda}$ and $m' \mathfrak{g}_{\lambda} \cap \mathfrak{g}_{\lambda} \neq \{0\}$ by the assumption.
By Lemma \ref{lem:Weyl_invariant}, 
$w$ can be written by 
\[
w = s_1 s_2 \cdots s_l
\]
where $s_i$ are root reflections of $W(\mathfrak{g},\mathfrak{a})$ with $s_i \lambda = \lambda$.
We write $\xi_i$ for the root of $\Sigma(\mathfrak{g},\mathfrak{a})$ corresponding to $s_i$ for each $i=1,\dots,l$. 
Let $\mathfrak{g}_i$ be the root space of $\xi_i$.
Since $s_i \lambda = \lambda$, each $\xi_i$ is orthogonal to $\lambda$ in $\mathfrak{a}^*$.
We can and do chose $X_i$ be a non-zero root vector of $\mathfrak{g}_i$ such that \[
B_\C(X_i,\theta X_i) = -\frac{2}{\langle \xi_i, \xi_i \rangle}
\]
where $\theta$ is the Cartan involution of $\mathfrak{g}$ 
corresponding to $\mathfrak{g} = \mathfrak{k} + \mathfrak{p}$.
Then the element $k_i = \exp \frac{\pi}{2} (X_i + \theta X_i)$ in $N_K(\mathfrak{a})$ acts on $\mathfrak{a}$ as the reflection $s_i$.
Thus $m := m' k_l k_{l-1} \cdots k_1$ acts trivially on $\mathfrak{a}$. 
That is, $m \in Z_K(\mathfrak{a}) = M$.
It remains to prove that $k_i X_\lambda = X_\lambda$. 
Since $\lambda$ is longest root of $\Sigma(\mathfrak{g},\mathfrak{a})$ 
by Lemma \ref{lem:restricted_highest_root}, 
and $\xi_i$ is orthogonal to $\lambda$, 
the element $\xi_i \pm \lambda$ of $\mathfrak{a}^*$ is not a root of $\Sigma(\mathfrak{g},\mathfrak{a})$.
In particular, $[X_i, X_\lambda] = 0$ and $[\theta X_i, X_\lambda] = 0$.
Hence, $k_i X_\lambda = X_\lambda$ for any $i$.
Therefore, we obtain that $m X_\lambda = m' X_\lambda$. 
\end{proof}

\subsection{$MA$-orbits in $\mathfrak{g}_\lambda$ in the cases where $\dim_\mathbb{R} \mathfrak{g}_\lambda \geq 2$}

In this subsection, we focus on the cases where $\dim_\mathbb{R} \mathfrak{g}_\lambda >2$, 
i.e.~$\mathfrak{g}$ is isomorphic to one of $\mathfrak{su}^*(2k)$, $\mathfrak{so}(n-1,1)$, $\mathfrak{sp}(p,q)$, $\mathfrak{e}_{6(-26)}$ or $\mathfrak{f}_{4(-20)}$,
and give a proof of Proposition \ref{prop:MA-orbits_in_g_lambda_2}.

We write $M_0$ for the identity component of $M$.
Then $M_0$, $M_0A$ are the analytic subgroups of $G$ with its Lie algebra 
$\mathfrak{m} = Z_{\mathfrak{k}}(\mathfrak{a})$, 
$\mathfrak{m} \oplus \mathfrak{a} = Z_{\mathfrak{g}}(\mathfrak{a})$, respectively.

Then the next lemma holds:

\begin{lem}\label{lem:real_rank_one_single_MA_orbit}
Suppose that $\dim_\mathbb{R} \mathfrak{g}_{\lambda} \geq 2$ 
and $\mathfrak{g}$ has real rank one, 
i.e.~$\dim_\mathbb{R} \mathfrak{a} = 1$.
Then $\mathfrak{g}_\lambda \setminus \{0\}$ becomes a single $M_0A$-orbit.
\end{lem}

\begin{proof}[Proof of Lemma $\ref{lem:real_rank_one_single_MA_orbit}$]
Let $A_{\lambda}$ be the coroot of $\lambda$ in $\mathfrak{a}$.
Since $\mathfrak{g}$ has real rank one, 
$\mathfrak{a} = \mathbb{R} A_{\lambda}$ and $\mathfrak{g}$ can be written by 
\[
\mathfrak{g} = \mathfrak{g}_{-\lambda} \oplus \mathfrak{g}_{-\frac{\lambda}{2}} \oplus \mathfrak{m} \oplus \mathfrak{a} \oplus \mathfrak{g}_{\frac{\lambda}{2}}\oplus \mathfrak{g}_{\lambda}
\]
(possibly $\mathfrak{g}_{\pm \frac{\lambda}{2}} = \{0\}$).
Let us denote by $\mathfrak{g}_\mathbb{C}$, $\mathfrak{m}_\mathbb{C}$, $\mathfrak{a}_\mathbb{C}$, $(\mathfrak{g}_{\pm \lambda})_\mathbb{C}$, $(\mathfrak{g}_{\pm \frac{\lambda}{2}})_\mathbb{C}$ the complexification of $\mathfrak{g}$, $\mathfrak{m}$, $\mathfrak{a}$, $\mathfrak{g}_{\pm \lambda}$, $\mathfrak{g}_{\pm \frac{\lambda}{2}}$, respectively.
We set
\[
(\mathfrak{g}_\mathbb{C})_i = \{\, X \in \mathfrak{g}_\mathbb{C} \mid [A_{\lambda}, X] = i X \,\} 
\quad \text{for each } i \in \mathbb{Z}.
\]
Then 
\[
(\mathfrak{g}_\mathbb{C})_0 = \mathfrak{m}_\mathbb{C} \oplus \mathfrak{a}_{\mathbb{C}},\ (\mathfrak{g}_\mathbb{C})_{\pm 1} = (\mathfrak{g}_{\pm \frac{\lambda}{2}})_\mathbb{C},\ (\mathfrak{g}_\mathbb{C})_{\pm 2} = (\mathfrak{g}_{\pm \lambda})_\mathbb{C}.
\]
By Lemma \ref{lem:sl_2_in_g}, for any non-zero highest root vector $X_\lambda$ in $\mathfrak{g}_\lambda$, there exists $Y_\lambda \in \mathfrak{g}_{-\lambda}$ such that 
$( A_{\lambda}, X_\lambda,Y_\lambda )$ is an $\mathfrak{sl}_2$-triple in $\mathfrak{g}_\mathbb{C}$.
By the theory of representations of $\mathfrak{sl}(2,\mathbb{C})$, we obtain that $[X_\lambda, (\mathfrak{g}_\mathbb{C})_0] = (\mathfrak{g}_\mathbb{C})_2$.
In particular, \[
[\mathfrak{m} \oplus \mathfrak{a}, X_\lambda] = \mathfrak{g}_\lambda. 
\]
Therefore, for the $M_0A$-orbit $\Orbit^{M_0A}(X_\lambda)$ in $\mathfrak{g}_\lambda$ through $X_\lambda$, we obtain that 
\[
\dim_\mathbb{R} \Orbit^{M_0A}(X_\lambda)= \dim_\mathbb{R} \mathfrak{g}_\lambda.
\]
This means that the $M_0A$-orbit $\Orbit^{M_0A}(X_\lambda)$ is open in $\mathfrak{g}_\lambda$ for any $X_\lambda \in \mathfrak{g}_\lambda \setminus \{0\}$.
Recall that we are assuming that $\dim_\mathbb{R} \mathfrak{g}_\lambda \geq 2$. 
Then $\mathfrak{g}_\lambda \setminus \{0\}$ is connected.
Therefore, $\mathfrak{g}_\lambda \setminus \{0\}$ becomes a single $M_0A$-orbit.
\end{proof}

Let us give a proof of Proposition \ref{prop:MA-orbits_in_g_lambda_2} 
by using Lemma \ref{lem:real_rank_one_single_MA_orbit} as follows.

\begin{proof}[Proof of Proposition $\ref{prop:MA-orbits_in_g_lambda_2}$]
Let us put 
$\mathfrak{h}' := [\mathfrak{g}_{\lambda},\mathfrak{g}_{-\lambda}] \subset \mathfrak{m} \oplus \mathfrak{a}$.
Then $\mathfrak{g}' := \mathfrak{g}_{-\lambda} \oplus \mathfrak{h}' \oplus \mathfrak{g}_{\lambda}$ 
becomes a subalgebra of $\mathfrak{g}$
since $\pm 2\lambda$ is not a root.
We shall prove that $\mathfrak{g}'$ is a simple Lie algebra of real rank one without complex structure.

Let $\theta$ be the Cartan involution of $\mathfrak{g}$ corresponding to $\mathfrak{g} = \mathfrak{k} + \mathfrak{p}$.
Then $\mathfrak{h}'$ is $\theta$-stable.
Therefore, $\mathfrak{h}'$ can be written by $\mathfrak{h}' = \mathfrak{m}' \oplus \mathfrak{a}'$ with $\mathfrak{m}' \subset \mathfrak{m}$ and $\mathfrak{a}' \subset \mathfrak{a}$.
For any $X_\lambda \in \mathfrak{g}_\lambda$, $X_{-\lambda} \in \mathfrak{g}_{-\lambda}$ and $A \in \mathfrak{a}$,
we have  
\begin{align*}
B_\C([X_\lambda,X_{-\lambda}],A) 
	&= B_\C(X_\lambda, [X_{-\lambda}, A]) \\
	&= \lambda(A) B_\C(X_\lambda,X_{-\lambda}) \\
	&= B_\C(X_\lambda,X_{-\lambda}) \frac{\langle \lambda, \lambda \rangle}{2} B_\C(A_{\lambda},A).
\end{align*}
Thus $\mathfrak{a}'$ can be written by $\mathfrak{a}' = \mathbb{R} A_{\lambda}$ since $B_\C(\mathfrak{g}_\lambda,\mathfrak{g}_{-\lambda}) = \mathbb{R}$, 
where $A_{\lambda}$ is the coroot of $\lambda$ in $\mathfrak{a}$
and $B_\C$ is the Killing form on $\mathfrak{g}_\C$.
For each $\mathfrak{s} = \mathfrak{g}', \mathfrak{h}', \mathfrak{m}',\mathfrak{a}, \mathfrak{g}_{\pm \lambda}$, 
We denote by $\mathfrak{s}_\mathbb{C}$ the complexification of $\mathfrak{s}$.
Let us fix any non-zero ideal $\mathfrak{I}$ of the complex Lie algebra $\mathfrak{g}'_\mathbb{C}$, 
and we shall prove that $\mathfrak{I} = \mathfrak{g}'_\mathbb{C}$.

First, we show $\mathfrak{I} \cap \mathfrak{g}_{-\lambda} \neq \{0\}$.
To this, we only need to prove that $\mathfrak{I} \cap (\mathfrak{g}_{-\lambda})_{\mathbb{C}} \neq \{0\}$
because $\mathfrak{I}$ is closed under the multiple of $\sqrt{-1}$.
We take a non-zero element $X$ in $\mathfrak{I}$.
Then the element $X$ can be written by 
\[
X = X_{\mathfrak{m}'} + c A_{\lambda} + X_{\lambda} + X_{-\lambda} \quad (X_{\mathfrak{m}'} \in \mathfrak{m}'_\mathbb{C}, c \in \mathbb{C}, X_{\lambda} \in (\mathfrak{g}_{\lambda})_\mathbb{C}, X_{-\lambda} \in (\mathfrak{g}_{-\lambda})_{\mathbb{C}}).
\]
We now construct a non-zero element in $\mathfrak{I} \cap (\mathfrak{g}_{-\lambda})_\mathbb{C}$ dividing into the following cases:
\begin{description}
\item [The cases where $X_\lambda \neq 0$] 
In this case, we can assume that $X_\lambda \in \mathfrak{g}_{\lambda}$ by the same argument above.
Then by Lemma \ref{lem:sl_2_in_g}, there exists $Y_{\lambda} \in \mathfrak{g}_{-\lambda}$ such that $( A_{\lambda}, X_\lambda,Y_{\lambda} )$ becomes an $\mathfrak{sl}_2$-triple.
Recall that $- 2 \lambda$ is not a root of $\Sigma(\mathfrak{g},\mathfrak{a})$. 
Thus we have \[
[Y_\lambda, [Y_{\lambda}, X]] = -2 Y_{\lambda}
\]
and hence $Y_{\lambda} \in \mathfrak{I} \cap \mathfrak{g}_{-\lambda}$.
\item [The cases where $X_\lambda = 0$ and $c \neq 0$]
In this case, for any non-zero vector $Y$ in $\mathfrak{g}_{-\lambda}$,
\[
[Y,X] = [Y,X_{\mathfrak{m}'} + c A_{\lambda}] = [Y,X_{\mathfrak{m}'}] + 2c Y \in (\mathfrak{g}_{-\lambda})_\mathbb{C}
\] is not zero since $\ad_{\mathfrak{g}_\mathbb{C}} X_{\mathfrak{m}'}$ has no non-zero real eigen-value.
Thus $[Y,X]$ is a non-zero vector of $\mathfrak{I} \cap (\mathfrak{g}_{-\lambda})_{\mathbb{C}}$.
\item [The cases where $X_\lambda = 0$, $c = 0$ and $X_{\mathfrak{m}'} \neq 0$]
In this case, we can assume that $X_{\mathfrak{m}'}$ is in $\mathfrak{m}'$ 
by the same argument above,
and we shall show that $[\mathfrak{g}_{-\lambda},X_{\mathfrak{m}'}] \neq \{0\}$ in $\mathfrak{g}_{-\lambda}$.
Since $X_{\mathfrak{m}'} \neq 0$, we have $\mathfrak{m}' \neq \{0\}$ in this case.
We now assume that $[\mathfrak{g}_{-\lambda},X_{\mathfrak{m}'}]$ is zero.
Then \[
B_\C(\mathfrak{h}', X_{\mathfrak{m}'}) = B_\C([\mathfrak{g}_{\lambda},\mathfrak{g}_{-\lambda}],X_{\mathfrak{m}'}) = B_\C(\mathfrak{g}_\lambda, [\mathfrak{g}_{-\lambda},X_{\mathfrak{m}'}]) = \{0\}.
\]
In particular, $B_\C(\mathfrak{m}',X_{\mathfrak{m}'}) = \{0\}$.
This contradicts the non-degenerateness of $B$ on $\mathfrak{k}$.
\item [The cases where $X = X_{-\lambda}$]
In this case, $X \in \mathfrak{I} \cap (\mathfrak{g}_{-\lambda})_\mathbb{C}$.
\end{description}
Thus $\mathfrak{I} \cap (\mathfrak{g}_{-\lambda})_{\mathbb{C}} \neq \{0\}$ and hence $\mathfrak{I} \cap \mathfrak{g}_{-\lambda} \neq \{0\}$.

We fix non-zero element $Y_\lambda$ in $\mathfrak{I} \cap \mathfrak{g}_{-\lambda}$.
Then by using Lemma \ref{lem:sl_2_in_g}, we can find $X_\lambda \in \mathfrak{g}_{\lambda}$ such that $( A_{\lambda}, X_\lambda,Y_\lambda )$ becomes an $\mathfrak{sl}_2$-triple in $\mathfrak{g}$ 
(since we can find $X_\lambda \in \mathfrak{g}_\lambda$ such that $( -A_{\lambda}, Y_\lambda, X_\lambda )$ is an $\mathfrak{sl}_2$-triple in $\mathfrak{g}$ by Lemma \ref{lem:sl_2_in_g} for $\xi = -\lambda$).
Hence, $A_{\lambda}$ is in $\mathfrak{I}$, and this implies that $\mathfrak{g}_\lambda, \mathfrak{g}_{-\lambda} \subset \mathfrak{I}$.
Since $\mathfrak{h}' = [\mathfrak{g}_{\lambda}, \mathfrak{g}_{-\lambda}]$, we have $\mathfrak{I} = \mathfrak{g}'_\mathbb{C}$.
This means that $\mathfrak{g}'_\C$ is a complex simple Lie algebra.
Since $\mathfrak{g}'$ is $\theta$-stable,
$\theta|_{\mathfrak{g}'}$ is a Cartan decomposition of $\mathfrak{g}'$ and $\mathfrak{a}' =  \mathbb{R} A_{\lambda}$ is a maximally split abelian subspace of $\mathfrak{g}'$.
In particular, $\mathfrak{g}' = \mathfrak{m}' \oplus \mathfrak{a}' \oplus \mathfrak{g}_{\pm \lambda}$ 
is a root space decomposition of $\mathfrak{g}'$.
Therefore, $\mathfrak{g}'$ is a real simple Lie algebra of real rank one with $\dim \mathfrak{g}'_\lambda \geq 2$ such that its complexification $\mathfrak{g}'_\C$ is also simple.

We denote by $M'_0$ the analytic subgroup of $G$ with its Lie algebra $\mathfrak{m}'$ and put $A'=\Exp \mathbb{R} A_{\lambda}$.
Then by Lemma \ref{lem:real_rank_one_single_MA_orbit}, 
we obtain that $\mathfrak{g}_\lambda \setminus \{0\}$ becomes a single $M'_0A'$-orbit.
Since any adjoint $M'_0A'$-orbit is contained in an adjoint $M_0A$-orbit, $\mathfrak{g}_\lambda \setminus \{0\}$ also becomes a single $M_0A$-orbit.
\end{proof}

\subsection{$MA$-orbits in $\mathfrak{g}_\lambda$ in the cases where $\dim_\mathbb{R} \mathfrak{g}_\lambda = 1$}

Throughout this subsection, we consider the cases where $\dim_\mathbb{R} \mathfrak{g}_\lambda = 1$ and give a proof of Proposition \ref{prop:MA-orbits_in_g_lambda_1}.

Let us denote by $\mathfrak{g}_\lambda^{+}$ and $\mathfrak{g}_{\lambda}^{-}$ the connected components of $\mathfrak{g}_\lambda \setminus \{0\}$.
Since for any $t \in \mathbb{R}$ and $X_\lambda \in \mathfrak{g}_\lambda$, 
\[
(\exp t A_{\lambda}) X_\lambda = e^{2t} X_\lambda,
\]
where $A_{\lambda}$ is the coroot of $\lambda$ in $\mathfrak{a}$,
$A = \exp \mathfrak{a}$ acts transitively on $\mathfrak{g}_\lambda^{+}$, $\mathfrak{g}_\lambda^{-}$, respectively.

We ask what is the condition to the existence of $m \in M$ such that $m \cdot \mathfrak{g}_{\lambda}^+ = \mathfrak{g}_\lambda^{-}$.
The following lemma answers our question:

\begin{lem}\label{lem:minus_in_F}
There exists $m \in M$ such that $m \cdot \mathfrak{g}_{\lambda}^+ = \mathfrak{g}_{\lambda}^{-}$ if and only if the type of the restricted root system $\Sigma(\mathfrak{g},\mathfrak{a})$ is not $C$ nor $BC$. 
Here, we consider the root system of type $A_1$, $B_2$ as $C_1$, $C_2$, respectively.
\end{lem}

To prove Lemma \ref{lem:minus_in_F}, 
we use the following fact for a structure of $M$.

\begin{fact}[cf.~{\cite[Chapter VII, Section 5]{Knapp02}}]\label{fact:str_of_M}
For any root $\xi$ of $\Sigma(\mathfrak{g},\mathfrak{a})$, 
we define $\gamma_\xi \in G_\mathbb{C}$ by 
\[
\gamma_\xi = \exp \pi \sqrt{-1}  A_{\xi},
\]
where $A_{\xi}$ is the coroot of $\xi$ in $\mathfrak{a}$. 
Let $F$ be the subgroup of $G_\mathbb{C}$ generated by $\gamma_\xi$ for all root $\xi$ of $\Sigma(\mathfrak{g},\mathfrak{a})$.
Then all $\gamma_\xi$ are in $M$ and $M = FM_0$, 
where $M_0$ is the identity component of $M$.
\end{fact}

\begin{proof}[Proof of Lemma \ref{lem:minus_in_F}]
We first assume that the type of $\Sigma(\mathfrak{g},\mathfrak{a})$ is not $C$  and not $BC$.
Then $\Sigma(\mathfrak{g},\mathfrak{a})$ is reduced and the Dynkin diagram of $\Sigma(\mathfrak{g},\mathfrak{a})$ satisfying the following property:
All nodes of the diagram corresponding to a longest root of $\Sigma(\mathfrak{g},\mathfrak{a})$ have some edges with odd multiplicity.
This means that for any longest root $\mu$, there exists a root $\xi$ of $\Sigma(\mathfrak{g},\mathfrak{a})$ such that 
$2\langle \mu, \xi \rangle/\langle \xi,\xi \rangle$
is odd.
In particular, since the highest root $\lambda$ of $\Sigma(\mathfrak{g},\mathfrak{a})$ is a longest root (by Lemma \ref{lem:restricted_highest_root}), we can find a root $\xi$ of $\Sigma(\mathfrak{g},\mathfrak{a})$ such that 
$2\langle \lambda, \xi \rangle/\langle \xi,\xi \rangle$
is odd.
Therefore, the element $\gamma_\xi = \exp \pi \sqrt{-1} A_{\xi}$ of $M$ (by Fact \ref{fact:str_of_M}) acts on $\mathfrak{g}_{\lambda}$ as the scalar multiplication of $-1$.
Thus in this case, we can take $m = \gamma_\xi$ satisfying that $m \cdot \mathfrak{g}_{\lambda}^{+} = \mathfrak{g}_{\lambda}^{-}$.
Conversely, we suppose that the type of $\Sigma(\mathfrak{g},\mathfrak{a})$
is $C$ or $BC$.
Then we can observe that for any longest root $\mu$ and root $\xi$ of $\Sigma(\mathfrak{g},\mathfrak{a})$, 
$
2\langle \mu, \xi \rangle/\langle \xi,\xi \rangle
$
is even.
Since the highest root $\lambda$ is longest, all generators $\gamma_\xi = \exp \pi \sqrt{-1} A_{\xi}$ of $F$ act on $\mathfrak{g}_\lambda$ trivially.
Thus all elements of $M = FM_0$ preserve 
$\mathfrak{g}_\lambda^{+}$ and $\mathfrak{g}_\lambda^{-}$, respectively.
\end{proof}

By the list of non-compact simple Lie algebras and its restricted root systems, we can obtain the following fact:
\begin{fact}\label{fact:list_of_C_BC}
Suppose that $\mathfrak{g}$ is a non-compact real simple Lie algebra with $\dim_\mathbb{R} \mathfrak{g}_\lambda =1$ $($thus, $\mathfrak{g}$ is not isomorphic to one of $\mathfrak{su}^*(2k)$, $\mathfrak{so}(n-1,1)$, $\mathfrak{sp}(p,q)$, $\mathfrak{e}_{6(-26)}$ nor $\mathfrak{f}_{4(-20)}$$)$.
Then the type of the restricted root system of $\mathfrak{g}$ is $C$ or $BC$ if and only if $(\mathfrak{g},\mathfrak{k})$ is Hermitian.
Here, we consider the root system of type $A_1$, $B_2$ as $C_1$, $C_2$, respectively.
\end{fact}

Combining Lemma \ref{lem:minus_in_F} with Fact \ref{fact:list_of_C_BC}, 
we obtain Proposition \ref{prop:MA-orbits_in_g_lambda_1}. 

\section*{Acknowledgements.}
The author would like to give warm thanks to Toshiyuki Kobayashi 
for suggesting the problem 
and whose comments were of inestimable value 
for this paper.

\def\cprime{$'$}
\providecommand{\bysame}{\leavevmode\hbox to3em{\hrulefill}\thinspace}
\providecommand{\MR}{\relax\ifhmode\unskip\space\fi MR }
\providecommand{\MRhref}[2]{%
  \href{http://www.ams.org/mathscinet-getitem?mr=#1}{#2}
}
\providecommand{\href}[2]{#2}

\end{document}